\documentclass[12pt]{amsart}

\usepackage{cases}
\usepackage[english]{babel}
\usepackage{times,bm,amsfonts,amsmath,amssymb,dsfont} 
\usepackage{graphicx} 
\usepackage[babel=true]{csquotes}
\usepackage{mathabx}
\usepackage{pdfsync}

\makeatletter
\def\section{\@startsection{section}{1}\z@{.9\linespacing\@plus\linespacing}%
  {.7\linespacing} {\fontsize{13}{15}\selectfont\scshape\centering}}
\def\paragraph{\@startsection{paragraph}{4}%
  \z@{0.3em}{-.5em}%
  {$\bullet$ \ \normalfont\itshape}}
\makeatother

\newtheorem{theo}{Theorem}[section]
\newtheorem{prop}[theo]{Proposition}
\newtheorem{lem}[theo]{Lemma}
\newtheorem{cor}[theo]{Corollary}

\theoremstyle{definition}

\newtheorem{notation}[theo]{Notation}

\theoremstyle{remark}
\newtheorem{rem}[theo]{Remark}
\newcommand\got[1]{{\bm{\mathfrak{#1}}}}
\makeatletter

\@addtoreset{equation}{section}
\makeatother

\usepackage{color}

\definecolor{gr}{rgb}   {0.,   0.69,   0.23 }
\definecolor{bl}{rgb}   {0.,   0.5,   1. }
\definecolor{mg}{rgb}   {0.85,  0.,    0.85}
\definecolor{yl}{rgb}   {0.8,  0.7,   0.}

\definecolor{webred}{rgb}{0.75,0,0}
\definecolor{webgreen}{rgb}{0,0.75,0}
\usepackage[citecolor=webgreen,colorlinks=true,linkcolor=webred]{hyperref}

\renewcommand{\d}{\, {\rm d}}

\newcommand{\R}{\mathbb{R}}

\newcommand{\Rp}{\R_{+}}
\newcommand{\Hplow}{\mathcal{H}^{-}}
\newcommand{\Hslow}{\Pi^{-}}
\newcommand{\Hpup}{\mathcal{H}^{+}}
\newcommand{\Hsup}{\Pi^{+}}
\newcommand{\Hs}{\Pi}

\newcommand{\spectre}{\lambda}

\newcommand{\spec}{\got{S}}
\newcommand{\specess}{\got{S}_{\mathrm{ess}}}

\newcommand{\dom}{\operatorname{Dom}}
\newcommand{\supp}{\operatorname{Supp}}

\newcommand{\curl}{\operatorname{curl}}
\newcommand{\dist}{\operatorname{dist}}

\newcommand{\sinc}{\operatorname{sinc}}

\newcommand{\Add}{{\bf A}^{\parallel}}

\newcommand{\uB}{{\bf B}^{\perp}}
\newcommand{\bB}{{\bf B}}
\newcommand{\bA}{{\bf A}}
\newcommand{\bt}{{\bf t}}

\newcommand{\dS}{{\mathbb{S}}}

\newcommand\cC{\mathcal{C}}
\newcommand\cW{\mathcal{W}}

\newcommand\cS{\mathcal{S}}
\newcommand\cE{\mathcal{E}}
\newcommand\cF{\mathcal{F}}

\renewcommand\gg{\got{g}}
\newcommand\gq{\got{q}}

\newcommand{\mudg}{\mu}

\newcommand{\Secteur}{\mathcal{S}}

\newcommand{\Los}{\rm R}

\newcommand{\sse}{s}
\newcommand{\ssess}{s_{\, \rm ess}}
\newcommand{\FQ}{\mathcal{Q}}
\newcommand{\FQpol}{\widetilde{\mathcal{Q}}}

\newcommand{\Vpol}{\widetilde{V}}

\newcommand{\sinf}{\sse^{\infty}}

\title[Magnetic Laplacian on wedges]{\large The model magnetic Laplacian on wedges}
\author{Nicolas Popoff}
\address{Laboratoire IRMAR, UMR 6625 du CNRS, Campus de Beaulieu, 35042 Rennes cedex, France
\\
 {\it E-mail address:} nicolas.popoff@univ-rennes1.fr}

\date{\today}

\begin{document}
\begin{abstract}
We study a model Schr\"odinger operator with constant magnetic field on an infinite wedge with natural boundary conditions. This problem is related to the semiclassical magnetic Laplacian on 3d domains with edges. We show that the ground energy is lower than the one coming from the regular part of the wedge and is continuous with respect to the geometry. We provide an upper bound for the ground energy on wedges of small opening. Numerical computations enlighten the theoretical approach.
\end{abstract} 
\maketitle

\section{Introduction}

\subsection{The magnetic Laplacian on model domains}
\label{S:motivation}
\paragraph{Motivation from the semiclassical problem}
Let $(-ih\nabla-{\bf A})^2$ be the magnetic Schr\"odinger operator (also called the magnetic Laplacian) on an open simply connected subset $\Omega$ of $\R^3$. The magnetic potential ${\bf A}:\R^3\mapsto \R^3 $ satisfies $\curl {\bf A}={\bf B}$ where ${\bf B}$ is a regular magnetic field and $h>0$ is a semiclassical parameter. For $\Omega$ bounded with Lipschitz boundary, the operator $(-ih\nabla-{\bf A})^2$ assorted with its natural Neumann boundary condition is an essentially self-adjoint operator with compact resolvent. Due to gauge invariance, the spectrum depends on $\bA$ only through the magnetic field $\bB$. We denote by $\spectre({\bf B},\Omega,h)$ the first eigenvalue of $(-ih\nabla-{\bf A})^2$.
   
 Many works have been dedicated to understanding the influence of the geometry (given by the domain $\Omega$ and the magnetic field ${\bf B}$) on the asymptotics of $\spectre({\bf B},\Omega,h)$ and on the localization of the associated eigenfunctions in the semiclassical limit $h\to0$. In dimension 2, let us cite the works \cite{BeSt, LuPan99-2, HeMo01, FouHel06} when $\Omega$ is regular and \cite{Bon06,BonDau06, BoDauMaVial07} when $\Omega$ is polygonal. In dimension 3, the regular case is studied in \cite{LuPan00, HeMo04, Ray3d09}.  When the boundary of $\Omega\subset\R^3$ has singularities, only few particular cases have been treated (\cite{Pan02, PofRay13}). 
  
 In order to find the main term of the asymptotics of $\spectre({\bf B},\Omega,h)$, we are led to study the magnetic Laplacian without semiclassical parameter ($h=1$) on unbounded model domains with a constant magnetic field. More precisely to each point $x\in \overline{\Omega}$ we associate its tangent cone $\Pi_{x}$, for example if $x$ belongs to the regular boundary of $\Omega$, $\Pi_{x}$ is a half-space. We also introduce $\bB_{x}:=\bB(x)$ the magnetic field ``frozen" at $x$ and $\bA_{x}$ an associated linear potential satisfying $\curl \bA_{x}=\bB_{x}$. We define
$$ H({\bf A}_{x}, \, \Pi_{x}):=(-i\nabla-{\bf A}_{x})^2$$
the tangent magnetic Laplacian on the model domain $\Pi_{x}$ with its natural Neumann boundary condition. We denote by
\begin{equation}
\label{D:spectremodel}
  E({\bf B}_{x},\Pi_{x}) \ \ \mbox{ the bottom of the spectrum of} \ \ H({\bf A}_{x}, \, \Pi_{x}) \ .
\end{equation}
When $\Omega$ is polyhedral \footnote{This analysis has its interest for a more general class of corner domains described for example in \cite[Section 2]{Dau88}.} and
 if the magnetic field does not vanish, we expect that $\spectre({\bf B},\Omega,h)$ behaves like $h \inf_{x\in\overline{\Omega}}E({\bf B}_{x},\Pi_{x})$ at first order when $h\to0$. A work with M. Dauge and V. Bonnaillie-No\"el is in progress to prove this rigorously with an estimate of the remainder (\cite{BoDauPof13}). Therefore we are led to find the points $x\in \overline{\Omega}$ whose tangent model problem minimizes $E(\bB_{x},\Pi_{x})$, in particular it is natural to investigate to continuity properties of $x\mapsto E(\bB_{x},\Pi_{x})$ on $\overline{\Omega}$. It is known that the restriction of this application to the regular part of the boundary of $\Omega$ is continuous. When $x$ is in the singular part of the boundary of $\Omega$ (edge or corner), it has already been proved for particular cases that $E(\bB_{x},\Pi_{x})$ can be strictly lower that the ground energy of model problems associated to  regular points close to $x$ (faces). In this article we investigate the model problems associated to edges: When $x$ belongs to an edge, we compare $E(\bB_{x},\Pi_{x})$ and the other spectral model quantities in a neighborhood of $x$.

\paragraph{The magnetic Laplacian on wedges}
The tangent cone to an edge is an infinite wedge. Let us denote by $(x_{1},x_{2},x_{3})$ the cartesian coordinates of $\R^3$. Let $\alpha\in(0,\pi)\cup(\pi,2\pi)$ be the opening angle, we denote by $\cW_{\alpha}$ the model wedge of opening $\alpha$: 
\begin{equation}
\label{D:Walpha}
\cW_{\alpha}:=\Secteur_{\alpha}\times \R 
\end{equation}
where 
$\Secteur_{\alpha}$ is the infinite sector defined by $\{ (x_{1},x_{2})\in \R^2, \, |x_{2}| \leq x_{1}\tan\tfrac{\alpha}{2}\}$ when $\alpha\in (0,\pi)$ and $\{ (x_{1},x_{2})\in \R^2, \, |x_{2}| \geq x_{1}\tan\tfrac{\alpha}{2} \}$ when $\alpha\in (\pi,2\pi)$. We extend these notations by using $\cW_{\pi}$ (respectively $\cS_{\pi}$) for the model half-space (respectively the model half-plane). For $\alpha\neq \pi$ the $x_{3}$-axis defines the edge of $\cW_{\alpha}$.

 Due to an elementary scaling, it is sufficient to consider unitary magnetic fields when dealing with tangent model problems, therefore in the following the magnetic field will always be constant and unitary. Let $\bB\in \dS^2$ and $\bA$ be an associated linear potential satisfying $\curl \bA=\bB$. In this article we investigate the bottom of the spectrum of the operator $H(\bA,  \cW_{\alpha})$ and the influence of the geometry given by $(\bB,\alpha)$. This operator has already been introduced in particular cases (see subsection \ref{S:statesoftheart}). Our results recover some of these particular cases in a more general context.

\subsection{Problematics}
\label{SS:modelproblems}
\paragraph{Singular chains}
For $\alpha\neq \pi$, the wedge $\cW_{\alpha}$ is a cone of $\R^3$ with singular chains (also called {\it strata}) corresponding to its structure far from its edge (see \cite{MazyaPlamenevskii77} or \cite{Dau88}). There are three singular chains: The half-space $\Hsup_{\alpha}$ corresponding to the upper face, the half-space  $\Hslow_{\alpha}$ corresponding to the lower face and the space $\R^3$ corresponding to interior points. When $\alpha\in(0,\pi)$ (convex case) we have $\Hsup_{\alpha}=\{ (x_{1},x_{2},x_{3})\in \R^3, \, x_{2} \leq x_{1}\tan\tfrac{\alpha}{2} \}$ and $\Hslow_{\alpha}=\{ (x_{1},x_{2},x_{3})\in \R^3, \, x_{2} \geq -x_{1}\tan\tfrac{\alpha}{2} \}$. Similar expressions can be found for $\alpha\in(\pi,2\pi)$ (non convex case). 

When the model domain is a half-space ($\alpha=\pi$), there is only one singular chain: The whole space $\R^3$. The bottom of the spectrum of the magnetic Laplacian on $\R^3$ is well known: 
\begin{equation}
\label{E:landau}
E(\bB,\R^3)=1 \ . 
\end{equation} 
For $\alpha\neq\pi$ we introduce the spectral quantity 
\begin{equation}
\label{D:lambdastar}
E^{*}(\bB,\cW_{\alpha}):=\min\left\{E(\bB,\Hsup_{\alpha}),E(\bB,\Hslow_{\alpha}),E(\bB,\R^3)\right\} \ . 
\end{equation}
When $\alpha=\pi$, we let $E^{*}(\bB,\cW_{\pi}):=E(\bB,\R^3)=1$.
\paragraph{The operator on half-spaces}
Before describing the meaning of $E^{*}$, we recall known result about the magnetic Laplacian on half-spaces and we exhibit the influence of the geometry on $E^{*}(\bB,\cW_{\alpha})$ . Let $\Hs\subset \R^3$ be a half-space. The bottom of the spectrum of the magnetic Laplacian on $\Hs$ depends only on the unoriented angle between the magnetic field $\bB$ and the boundary of $\Hs$. We denote by $\theta\in [0,\frac{\pi}{2}]$ this angle. We denote by $\sigma(\theta)$ the bottom of the spectrum of the operator $H(\bA,\Hs)$. This function has already been studied in \cite{LuPan00}, \cite{HeMo02} or more recently \cite{BoDauPopRay12}. In particular $\theta\mapsto \sigma(\theta)$ is increasing over $[0,\frac{\pi}{2}]$ with $\sigma(0)=\Theta_{0}$ and $\sigma(\frac{\pi}{2})=1$ (see \cite{LuPan00}) where the universal constant $\Theta_{0}\approx 0.59$ is a spectral quantity associated to a unidimensional operator on a half-axis (see \cite{SaGe63, BolHe93, DauHe93} and Subsection \ref{SS:modelproblemregular}). 

Let us denote by $\theta^{+}$ (respectively $\theta^{-}$) the unoriented angle between the magnetic field $\bB$ and $\Hsup_{\alpha}$ (respectively $\Hslow_{\alpha}$). We have $E(\bB,\Hsup_{\alpha})=\sigma(\theta^{+})$, $E(\bB,\Hslow_{\alpha})=\sigma(\theta^{-})$ and $E(\bB,\R^3)=1$. Since $\sigma$ is increasing we get 
\begin{equation}
\label{E:lambdastarexplicit}
E^{*}(\bB,\cW_{\alpha})=\sigma(\min\{\theta^{+},\theta^{-}\}) \ . 
\end{equation}

\paragraph{Main goals and consequences}
When $\alpha\neq \pi$, the quantity $E^{*}(\bB,\cW_{\alpha})$ can be interpreted as the lowest energy of the magnetic Laplacian far from the $x_{3}$-axis. Remind the semiclassical problem on a bounded domain $\Omega$ with edges described in Subsection \ref{S:motivation}. If $x\in \overline{\Omega}$ belongs to an edge and if $\cW_{\alpha}$ is the tangent cone to $\Omega$ at $x$, the quantity $E^{*}(\bB,\cW_{\alpha})$ corresponds to the lowest energy among the $E(\bB_{y},\Pi_{y})$ for all regular points $y$ near $x$.  In this article we prove
\begin{equation}
\label{E:ineqwedge}
\forall \alpha\in (0,2\pi), \quad E(\bB,\cW_{\alpha}) \leq E^{*}(\bB,\cW_{\alpha}) \ ,
\end{equation}
roughly speaking that means that the ground energy associated to an edge is lower than the one of regular adjacent model problems.
\begin{rem}
When $\alpha=\pi$, we have $\theta^{-}=\theta^{+}=\theta$ and $E(\bB,\cW_{\pi})=\sigma(\theta)$. Since $\sigma(\theta)\leq1$ with equality if and only if $\theta=\frac{\pi}{2}$, we notice that \eqref{E:ineqwedge} is already known for $\alpha=\pi$ with equality if and only if $\bB$ is normal to the boundary of the half-space $\cW_{\pi}$.
\end{rem}

 Relation \eqref{E:ineqwedge} may either be strict or be an equality. It has been shown on examples that both cases are possible, see Subsection \ref{S:statesoftheart}. When $\eqref{E:ineqwedge}$ is strict the singularity makes the energy lower than in the regular cases close to the edge, moreover this would bring more precise asymptotics and localization properties for the lowest eigenpairs of the magnetic Laplacian in the semiclassical limit (see \cite{Pan02}, \cite{Bon06}, \cite{BonDau06} and \cite{PofRay13}). In Section \ref{S:HL} we will give a generic geometrical condition for which \eqref{E:ineqwedge} is strict.

As we will see, the operator $H(\bA, \cW_{\alpha})$ is fibered and reduces to a family of 2d operators after Fourier transform along the axis of the wedge. We will Link $E^{*}(\bB,\cW_{\alpha})$ and spectral quantities associated to the reduced operator family. As a consequence when the inequality \eqref{E:ineqwedge} is strict, this will provide existence of {\em generalized} eigenpairs for $H(\bA, \cW_{\alpha})$ with energy $E(\bB,\cW_{\alpha})$, moreover these generalized eigenfunctions are localized near the edge (see Corollary \ref{C:generalizedef}). 

We will also prove the continuity of $(\bB,\alpha)\mapsto E(\bB,\cW_{\alpha})$ on $\dS^2\times (0,2\pi)$. Let us remark that the continuity is proven even for the degenerate case $\alpha=\pi$. When $\Omega$ is polyhedral, this fact and \eqref{E:ineqwedge} are key ingredients in order to prove that $x\mapsto E(\bB_{x},\Pi_{x})$ is lower semi-continuous on $\overline{\Omega}$ (see \cite{BoDauPof13}) in order to find a minimizer for $E(\bB_{x},\Pi_{x})$ on $\overline{\Omega}$. 

\subsection{State of the art on wedges}
\label{S:statesoftheart}
The model operator on infinite wedges has already been explored for particular cases:

 In \cite{Pan02}, X. B. Pan studies the case of wedges of opening $\frac{\pi}{2}$ and applies his results to the semiclassical problem on a cuboid. In particular he shows that the inequality \eqref{E:ineqwedge} is strict if the magnetic field is tangent to a face of the wedge but not to the axis. These results can hardly be extended to the general case.
 
 The case of the magnetic field $\bB_{0}:=(0,0,1)$ tangent to the edge reduces to a magnetic Laplacian on the sector $\cS_{\alpha}$. This case is studied in \cite{Bon06} (see also \cite{Ja01} for $\alpha=\frac{\pi}{2}$): There holds $E^{*}(\bB,\cW_{\alpha})=\sigma(0)=\Theta_{0}$ and it is proved that the inequality \eqref{E:ineqwedge} is strict at least for $\alpha\in (0,\frac{\pi}{2}]$. V. Bonnaillie shows in particular that $E(\bB,\cW_{\alpha})\sim\frac{\alpha}{\sqrt{3}}$ when $\alpha\to0$ and gives a complete expansion of $E(\bB,\cW_{\alpha})$ in power of $\alpha$. 
 
  In \cite{Pof13T}, the author considers a magnetic field tangent to a face of the wedge. In that case he proves \eqref{E:ineqwedge} with $E^{*}(\bB,\cW_{\alpha})=\Theta_{0}$. He shows that the inequality \eqref{E:ineqwedge} is strict for $\alpha$ small enough but he also exhibits cases of equality. 
  
  In \cite{PofRay13}, the authors deal with a magnetic field normal to the plane of symmetry of the wedge and show that \eqref{E:ineqwedge} is strict at least for $\alpha$ small enough.

The results of this article recover these particular cases and give a more general approach about the model problem on wedges. 

\subsection{Organization of the article}
In Section \ref{S:reductiontosector} we reduce the operator $H(\bA,\cW_{\alpha})$ to a family of fibers on the sector $\cS_{\alpha}$. In Section \ref{S:link}, we link the problem on the wedge with model operators on half-spaces corresponding to the two faces and we deduce \eqref{E:ineqwedge}. In section \ref{S:continuity} we prove that $E(\bB,\cW_{\alpha})$ is continuous with respect to the geometry given by $(\bB,\alpha)\in \dS^2\times (0,2\pi)$. In Section \ref{S:UbSa} we use a 1d operator to construct quasimodes for $\alpha$ small and we exhibit cases where the inequality \eqref{E:ineqwedge} is strict. In Section \ref{S:numerique} we give numerical computation of the eigenpairs of the reduced operator on the sector.

\section{From the wedge to the sector}
\label{S:reductiontosector}
\subsection{Reduction to a sector}
\label{Reductiontoasector}
Due to the symmetry of the problem (see \cite[Proposition 3.14]{Popoff} for the detailed proof) we have the following:
\begin{prop} 
\label{P:symmetry}
Let $\bB=(b_{1},b_{2},b_{3})$ be a constant magnetic field and $\bA$ an associated potential. The operator $H(\bA, \cW_{\alpha})$ is unitary equivalent to $H({\widetilde{\bA}},\cW_{\alpha})$ where $\widetilde{\bA}$ satisfies $\curl \widetilde{\bA}=(|b_{1}|,|b_{2}|,|b_{3}|)$. 
\end{prop}
Therefore we can restrict ourselves to the case $b_{i}\geq0$.

We assume that the magnetic potential ${\bf A}=(a_{1},a_{2},a_{3})$ satisfies $\curl {\bf A}={\bf B}$ and the magnetic Schr\"odinger operator writes: 
$$H({\bf A}, \, \cW_{\alpha})=\sum_{j=1}^{3}(D_{x_{j}}-a_{j})^2$$
with $D_{x_{j}}=-i\partial_{x_{j}}$. Due to gauge invariance, the spectrum of $H({\bf A}, \, \cW_{\alpha})$ does not depend on the choice of ${\bf A}$ as soon as it satisfies $\curl {\bf A}={\bf B}$. Moreover we can choose $\bA$ independent of the $x_{3}$ variable. The magnetic field will be chosen explicitly later, see \eqref{D:A}.

We denote by $\spec(P)$ (respectively $\specess(P)$) the spectrum (respectively the essential spectrum) of an operator $P$. Due to the invariance by translation in the $x_{3}$-variable, the spectrum of $H({\bf A}, \, \cW_{\alpha})$ is absolutely continuous and we have $\spec(H({\bf A}, \, \cW_{\alpha}))=\specess(H({\bf A}, \, \cW_{\alpha}))$.

\subsubsection{Partial Fourier transform}
 Let $\tau\in \R$ be the Fourier variable dual to $x_{3}$ and $\cF_{x_{3}}$ the associated Fourier transform. We recall that $\bA$ has been chosen independent of the $x_{3}$ variable and we introduce the operator
 $$\widehat{H}^{\tau}(\bA,\cW_{\alpha}):=(D_{x_{1}}-a_{1})^2+(D_{x_{2}}-a_{2})^2+(a_{3}-\tau)^2 $$
 acting on $L^2(\cS_{\alpha})$ with natural Neumann boundary condition. We have the following direct integral decomposition (see \cite{ReSi78}): 
\begin{equation}
\label{E:decompintdir1}
\cF_{x_{3}}H({\bf A},\, \cW_{\alpha})\cF_{x_{3}}^{*} =\int_{\tau\in\R}^{\bigoplus} \widehat{H}^{\tau}(\bA,\cW_{\alpha}) \d \tau \ .
\end{equation}
The operator $H({\bf A},\, \cW_{\alpha})$ is a fibered operator (see \cite{GeNi98}) whose fibers are the 2d operators $\widehat{H}^{\tau}(\bA,\cW_{\alpha})$ with $\tau\in \R$. Let 
 $$\sse({\bf B},\cS_{\alpha};\tau):=\inf\spec(\widehat{H}^{\tau}(\bA,\cW_{\alpha}))$$ 
be the bottom of the spectrum of $\widehat{H}^{\tau}(\bA,\cW_{\alpha})$, also called the \em band function.\rm 
Thanks to \eqref{E:decompintdir1} we have the following fundamental relation: 
\begin{equation}
\label{E:relationfond}
E({\bf B},\cW_{\alpha})=\inf_{\tau\in\R} \sse({\bf B},\cS_{\alpha};\tau) \ .
\end{equation}
As a consequence we are reduced to study the spectrum of a 2d family of Schr\"odinger operators. We denote by
$$ \ssess(\bB,\cS_{\alpha};\tau):=\inf \specess(\widehat{H}^{\tau}(\bA,\cW_{\alpha})) \  $$
the bottom of the essential spectrum.

\subsubsection{Description of the reduced operator}
We write $$\bB=\uB+\bB^{\parallel}$$
where $\uB=(b_{1},b_{2},0)$ and $\bB^{\parallel}=(0,0,b_{3})$. We take for the magnetic potential
\begin{equation}
\label{D:A}
{\bf A}(x_{1},x_{2},x_{3})=(\Add(x_{1},x_{2}),a^{\perp}(x_{1},x_{2}))
\end{equation}
 with $\Add(x_{1},x_{2}):=(0,b_{3}x_{1})$ and $a^{\perp}(x_{1},x_{2})=x_{2}b_{1}-x_{1}b_{2}$. The magnetic potentiel $\bA$ is linear, does not depend on $x_{3}$ and satisfies $\curl \bA=\bB$. We introduce the reduced electric potential on the sector: 
$$  V_{\uB}^{\tau}(x_{1},x_{2}):=(x_{1}b_{2}-x_{2}b_{1}-\tau)^2 \ .$$
We have 
\begin{equation}
\label{}
\widehat{H}_{\tau}(\bA,\cW_{\alpha})=H(\Add,\cS_{\alpha})+V_{\uB}^{\tau} \ .
\end{equation}
The quadratic form of $H(\Add,  \Secteur_{\alpha})+ V_{\uB}^{\tau}$ is 
$$ \FQ_{\bB,\alpha}^{\tau}(u):=\int_{\Secteur_{\alpha}}|(-i\nabla-\Add)u|^2+V_{\uB}^{\tau}|u|^2 \d x_{1} \d x_{2}  $$
defined on the form domain
\begin{equation}
\label{D:DomFQ}
\dom(\FQ_{\bB,\alpha}^{\tau})= \{u\in L^2(\Secteur_{\alpha}), \, (-i\nabla-\Add)u\in L^2(\Secteur_{\alpha}), \, |x_{1}b_{2}-x_{2}b_{1}-\tau|u\in L^2(\Secteur_{\alpha}) \}  \ .
\end{equation}
The form domain coincides with: 
$$\{u\in L^2(\Secteur_{\alpha}), \, (-i\nabla-\Add)u\in L^2(\Secteur_{\alpha}), \, |x_{1}b_{2}-x_{2}b_{1}|u\in L^2(\Secteur_{\alpha}) \} \ , $$
therefore it does not depend on $\tau$. Kato's perturbation theory (see \cite{kato}) provides the following:
\begin{prop}
The function $\tau\mapsto\sse({\bf B},\cS_{\alpha};\tau)$ is continuous on $\R$. 
\end{prop}

\subsection{Model problems on regular domain}
\label{SS:modelproblemregular}

We describe here the case $\alpha=\pi$ where $\cW_{\pi}$ is a half-space. The operator $H(\Add,  \cS_{\pi})+V_{\uB}^{\tau}$ can be analyzed using known results about regular domain. We have $E(\bB,\cW_{\pi})=\sigma(\theta)$ (see Subsection \ref{SS:modelproblems} and \cite{HeMo02}) where $\theta\in [0,\frac{\pi}{2}]$ is the angle between the magnetic field and the boundary. We recall that we have $E^{*}(\bB,\cW_{\pi})=1$. 

When $\theta\neq0$, $H(\Add,  \cS_{\pi})+V_{\uB}^{\tau}$ is unitary equivalent to $H(\Add,  \cS_{\pi})+V_{\uB}^{0}$ and
 $\ssess(\bB,\cS_{\pi};0)=1$ (\cite[Proposition 3.4]{HeMo02}). There holds $\sse(\bB,\cS_{\pi};0)=\sigma(\theta)\leq1$. If $\theta\neq \frac{\pi}{2}$, $\sigma(\theta)<1$ and therefore the operator $H(\Add,  \cS_{\pi})+V_{\uB}^{0}$ has an eigenfunction associated to $\sigma(\theta)$ with exponential decay (see \cite{BoDauPopRay12}). 

When $\theta=0$, there holds $\ssess(\bB,\cS_{\pi};\tau)=\sse(\bB,\cS_{\pi};\tau)$. A partial Fourier transform can be performed and shows that $\inf_{\tau\in \R}\sse(\bB,\cS_{\pi};\tau)=\Theta_{0}$.

In Subsection \ref{SS:essential} and Section \ref{S:link} we will focus on $\alpha\in(0,\pi)\cup(\pi,2\pi)$. Most of the results can be compared and extended to $\alpha=\pi$ using the results recalled above.

\subsection{Link between the geometry and the essential spectrum of the reduced problem}
\label{SS:essential}
In this section  we give the essential spectrum of the operator $H(\Add,  \Secteur_{\alpha})+ V_{\uB}^{\tau}$ depending on the geometry.
 Let $\Upsilon:=(V_{\uB}^{\tau})^{-1}(\{0\})$ be the line where the electric potential vanishes. Let us notice that $V_{\uB}^{\tau}(x)$ is the square of the distance from $x$ to $\Upsilon$. Let $(\gamma,\theta)$ be the spherical coordinates of the magnetic field where $\gamma$ is the angle between the magnetic field and the $x_{3}$-axis and $\theta$ is the angle between the projection $(b_{1},b_{2})$ and the $x_{2}$-axis: 
$$ \bB=(\sin\gamma\sin\theta,\sin\gamma\cos\theta,\cos\gamma) \ .$$
Due to symmetries we restrict ourselves to $(\gamma,\theta)\in[0,\frac{\pi}{2}]\times[0,\frac{\pi}{2}]$. We will use the following terminology: 
\begin{itemize}
\item The magnetic field is outgoing if $\alpha\in (0,\pi)$ and $\theta\in[0,\frac{\pi-\alpha}{2})$. 
\item The magnetic field is tangent if either $\gamma=0$ or $\theta=\frac{|\pi-\alpha|}{2}$.
\item The magnetic field is ingoing in the other cases. 
\end{itemize}
The ongoing case corresponds to a magnetic field pointing outward the wedge (this can happen only if the wedge is convex). The tangent case corresponds to a magnetic field tangent to a face of the wedge and has already been explored for convex wedges in \cite{Pof13T}. The ingoing case corresponds to a magnetic field pointing inward the wedge, in that case the intersection between $\Upsilon$ and $\cS_{\alpha}$ is always unbounded. The essential spectrum of $H(\Add,  \Secteur_{\alpha})+ V_{\uB}^{\tau}$ depends on the situation as described below: 
\begin{prop}
Let $\alpha\in(0,\pi)$ and $\bB\in \dS^2$ be an outgoing magnetic field. Then for all $\tau\in\R$ the operator $H(\Add,  \Secteur_{\alpha})+ V_{\uB}^{\tau}$ has compact resolvent.
\end{prop}
\begin{proof}
We remark that 
$$\forall \tau \in \R, \quad \lim_{\underset{(x_{1},x_{2})\in \cS_{\alpha}}{|(x_{1},x_{2})|\to+\infty}}V_{\uB}^{\tau}(x_{1},x_{2})=+\infty \ .$$
This implies that the injection from the form domain \eqref{D:DomFQ} into $L^2(\cS_{\alpha})$ is compact, see for example \cite{ReSi78}. We deduce that the operator $H(\Add,  \Secteur_{\alpha})+ V_{\uB}^{\tau}$ has compact resolvent.
\end{proof}

The following proposition shows that the essential spectrum is much more different when the magnetic field is ingoing:
\begin{prop}
Let $\alpha\in(0,\pi)\cup(\pi,2\pi)$ and $\bB\in \dS^2$ be an ingoing magnetic field. Then 
$$\forall \tau\in \R, \quad \ssess(\bB,\cS_{\alpha};\tau)=1 \ . $$
\end{prop}
When $\alpha\in (0,\pi)$, the detailed proof can be found in \cite[Subsection 4.2.2]{Popoff}. The proof for $\alpha\in(\pi,2\pi)$ is rigorously the same. The idea is to construct a Weyl's quasimode for $\FQ_{\bB,\alpha}^{\tau}$ far from the origin and near the line $\Upsilon$ using the operator $H(\Add,\R^2)+V_{\uB}^{\tau}$ whose first eigenvalue is 1. The persson's lemma (see \cite{Pers60}) provides the result.

In the tangent case, the essential spectrum depends on the parameters and can be expressed using the first eigenvalue of the classical 1d de Gennes operator (see the proof below). The bottom of the essential spectrum is given explicitly in \eqref{F:infsurxi2} however we will only need the following:
\begin{prop}
\label{P:specess}
Let $\alpha\in(0,\pi)\cup(\pi,2\pi)$ and $\bB\in \dS^2$ be a magnetic field tangent to $\cW_{\alpha}$. Then we have 
$$\inf_{\tau\in\R}\ssess(\bB,\cS_{\alpha};\tau)=\Theta_{0} \ . $$
\end{prop}
\begin{proof}
We introduce the first eigenvalue $\mudg(\tau)$ of the 1d de Gennes operator
$$ -\partial_{t}^2+(t-\tau)^2$$ 
defined on the half-line $\{t>0\}$ with a Neumann boundary condition. This classical spectral quantity has already been investigated, see \cite{SaGe63, BolHe93,DauHe93}.  In particular $\mudg(\tau)$ reaches a unique minimum $\Theta_{0}\approx 0.59$ for $\xi_{0}=\sqrt{\Theta_{0}}$. We recall the result from \cite[Proposition 3.6]{Pof13T}: 
\begin{equation}
\label{F:infsurxi2}
\ssess({\bf B},\cS_{\alpha};\tau)=\inf_{\xi\in\R}\left(\mudg(\xi\cos\gamma+\tau\sin\gamma)+(\xi\sin\gamma-\tau\cos\gamma)^2\right) \ . 
\end{equation}
where $\gamma\in [0,\frac{\pi}{2}]$ is the angle between the magnetic field and the axis of the wedge. Note that the proof of this relation is done in \cite{Pof13T} for $\alpha\in (0,\pi)$ and the extension to $\alpha\in (\pi,2\pi)$ does not need any additional work. We deduce from \eqref{F:infsurxi2} that 
\begin{equation}
\label{E:compcastangentsess}
\forall \tau\in \R, \quad \ssess(\bB,\cS_{\alpha};\tau)\geq \Theta_{0}  \ . 
\end{equation}
Choosing $\xi=\xi_{0}\cos\gamma$ in \eqref{F:infsurxi2} we get $\ssess(\bB,\cS_{\alpha},\xi_{0}\sin\gamma)=\mu(\xi_{0})=\Theta_{0}$ and the proposition is proven.
\end{proof}

\begin{rem}
We have $\sigma(0)=\Theta_{0}$ where the function $\sigma$ is defined in Subsection \ref{SS:modelproblems}.
\end{rem}
Since $\sse(\bB,\cS_{\alpha};\tau) \leq \ssess(\bB,\cS_{\alpha};\tau)$, the relation \eqref{E:relationfond} provides for a tangent magnetic field:
\begin{equation}
\label{E:compcastangent}
\forall \alpha \in (0,2\pi)\setminus \pi, \quad E(\bB,\cW_{\alpha})\leq \Theta_{0} \ . 
\end{equation}
Therefore we have proved \eqref{E:ineqwedge} for a tangent magnetic field.

\section{Link with problems on half-planes}
\label{S:link}
In this section we will investigate the link between the model operator on a wedge of opening $\alpha\in(0,\pi)\cup(\pi,2\pi)$ and the model operator on the half-spaces $\Hsup_{\alpha}$, $\Hslow_{\alpha}$ and the space $\R^3$ (see Subsection \ref{SS:modelproblems}). These domains are the singular chains of $\cW_{\alpha}$. We recall that $E^{*}(\bB,\cW_{\alpha})$ is the lowest energy of the magnetic Laplacian $(-i\nabla-\bA)^2$ acting on these singular chains and is given by 
$$E^{*}(\bB,\cW_{\alpha})=\sigma(\min\{\theta^{+},\theta^{-}\})$$
where $\theta^{\pm}$ is the angle between $\bB$ and $\Pi_{\alpha}^{\pm}$ and $\sigma(\cdot)$ is defined in Subsection \ref{SS:modelproblems}.  In this section we prove the inequality \eqref{E:ineqwedge}. Moreover when this inequality is strict we show that the band function $\tau\mapsto \sse(\bB,\cS_{\alpha};\tau)$ reaches its infimum and that this infimum is a discrete eigenvalue for the reduced operator on the sector.  Let us remark that these questions were investigated in \cite{Pan02} and \cite{Pof13T} for particular cases.
 
We denote by $\Hpup_{\alpha}$ and $\Hplow_{\alpha}$ the half-planes such that $\Hsup_{\alpha}=\R\times \Hpup_{\alpha}$ and $\Hslow_{\alpha}=\R\times \Hplow_{\alpha}$.  Let $H(\Add,  \Hpup_{\alpha})+ V_{\uB}^{\tau}$ be the reduced operator defined on $\Hpup_{\alpha}$ with a Neumann boundary condition. When $\bB$ is not tangent to $\Hsup_{\alpha}$ we deduce from Subsection \ref{SS:modelproblemregular}: 
\begin{equation}
\label{E:infspecPup}
\forall \tau \in \R, \quad \inf \spec(H(\Add,  \Hpup_{\alpha})+ V_{\uB}^{\tau})=\sigma(\theta^{+})
\end{equation}
 Similarly when the magnetic field is not tangent to $\Hslow_{\alpha}$ we have:
\begin{equation}
\label{E:infspecPlow}
\forall \tau \in \R, \quad\inf \spec(H(\Add,  \Hplow_{\alpha})+ V_{\uB}^{\tau})=\sigma(\theta^{-})
\end{equation}

\subsection{Limits for large Fourier parameter}
In this section we investigate the behavior of $\sse(\bB,\cS_{\alpha};\tau)$ when the Fourier parameter $\tau$ goes to $\pm\infty$. We introduce the quantity 
\begin{equation}
\label{D:sinf}
\sinf(\bB,\cS_{\alpha}):=\min\left\{\liminf_{\tau\to-\infty}\sse(\bB,\cS_{\alpha};\tau),\liminf_{\tau\to+\infty}\sse(\bB,\cS_{\alpha};\tau) \right\} \ . 
\end{equation}
In the tangent case, we recall the results from \cite[Section 4]{Pof13T}: 
\begin{prop}
\label{P:liminfTcase}
Let $\alpha\in(0,\pi)\cup(\pi,2\pi)$ and let $\bB\in \dS^2$ be a magnetic field tangent to a face of the wedge $\cW_{\alpha}$. 
Then we have
$$\sinf(\bB,\cS_{\alpha})= \sigma(\max(\theta^{-},\theta^{+}))\ . $$
\end{prop}
Note that in \cite{Pof13T}, this result is proved only for $\alpha\in(0,\pi)$. The proof of \cite[Proposition 4.1]{Pof13T} is mimicked to the case $\alpha\in (\pi,2\pi)$. 

We recall the useful {\em IMS localization formula} (see \cite[Theorem 3.2]{CyFrKiSi87}): 
\begin{lem}
\label{L:IMS}
Let $(\chi_{j})$ be a finite regular partition of the unity satisfying $\sum\chi_{j}^2=1$. We have for $u\in \dom(\FQ_{\bB,\alpha}^{\tau})$:
$$\FQ_{\bB,\alpha}^{\tau}(u)=\sum_{j}\FQ_{\bB,\alpha}^{\tau}(\chi_{j}u)-\sum_{j}\| \nabla\chi_{j}u\|_{L^2}^2  \ .$$
\end{lem}

The following lemma gives a lower bound on the energy of a function supported far from the corner of the sector. This lemma will also be useful in Section \ref{S:continuity}. We denote by $B(0,R)$ the ball centered at the origin of radius $R>0$ and $\complement B(0,R)$ its complementary.

\begin{lem}
\label{L:estinfty1}
There exists $C_{1}>0$ and $R_{0}>0$ such that for all $\alpha\in (0,\pi)\cup(\pi,2\pi)$ and for all $\bB\in\dS^2$, for all $R \geq R_{0}$, for all $\tau\in\R$, for all $u\in \dom(\FQ_{\bB,\alpha}^{\tau})$ such that $\supp(u)\subset\complement \overline{B(0,R)}$: 
$$\FQ_{\bB,\alpha}^{\tau}(u) \geq  \left(E^{*}(\bB,\cW_{\alpha})-\frac{C_{1}}{\alpha^2R^2}\right)\| u \|^2_{L^2} \ .$$  
\end{lem}
\begin{proof}
Let $(\chi_{j})_{j=1,2,3}$ be a partition of unity satisfying $\chi_{j}\in C_{0}^{\infty}([-\frac{1}{2},\frac{1}{2}],[0,1])$, $\supp(\chi_{j})\subset [\frac{j-3}{4},\frac{j-1}{4}]$ and $\sum_{j}\chi_{j}^2=1$. We defined the cut-off functions $\chi_{j,\alpha}^{\rm pol}(r,\psi):=\chi_{j}(\frac{\psi}{\alpha})$ where $(r,\psi)\in \R_{+}\times (-\frac{\alpha}{2},\frac{\alpha}{2})$ are the polar coordinates. We denote by $\chi_{j,\alpha}$ the associated functions in cartesian coordinates. Since the $\chi_{j,\alpha}$ do not depend on $r$, there exists $C_{1}>0$ such that 
$$\forall \alpha \in (0,2\pi),\forall R>0, \quad \forall (x_{1},x_{2})\in \complement B(0,R), \quad \sum_{j=1}^{3}|\nabla\chi_{j,\alpha}(x_{1},x_{2})|^2 \leq \frac{C_{1}}{R^2\alpha^2} \ . $$
Let $u\in \dom \FQ_{\bB,\alpha}^{\tau}$ such that $\supp(u)\subset\complement B(0,R)$. The IMS formula (see Lemma \ref{L:IMS}) provides 
\begin{equation}
\label{IMS:partitionangle}
\FQ_{\bB,\alpha}^{\tau}(u) \geq \sum_{j=1}^{3}\FQ_{\bB,\alpha}^{\tau}(\chi_{j,\alpha}u) - \frac{C_{1}}{\alpha^2R^2}\| u \|_{L^2}^2 \ . 
\end{equation}
Moreover $\chi_{1}u$ and $\chi_{3}u$ are extended to functions of $L^2(\Hpup_{\alpha})$ and $L^2(\Hplow_{\alpha})$ with the suitable Neumann boundary conditions by setting $\chi_{j}u=0$ outside $\supp(\chi_{j})$. We deduce from \eqref{E:infspecPup} and the min-max principle that $\FQ_{\bB,\alpha}^{\tau}(\chi_{1,\alpha}u) \geq \sigma(\theta^{+}) \|\chi_{1,\alpha}u\|^2_{L^2}$. Similarly we prove $\FQ_{\bB,\alpha}^{\tau}(\chi_{3,\alpha}u) \geq \sigma(\theta^{-}) \|\chi_{3,\alpha}u\|^2_{L^2}$. The function $\chi_{2,\alpha}u$ is extended in the same way to a function of $\R^2$. It is elementary that 
$$\forall \tau \in \R, \quad\inf \spec(H(\Add,  \R^2)+ V_{\uB}^{\tau})=E(\bB,\R^3)=1 \ , $$
therefore $\FQ_{\bB,\alpha}^{\tau}(\chi_{2,\alpha}u) \geq \|\chi_{2,\alpha}u\|^2_{L^2}$. We conclude with \eqref{IMS:partitionangle} and the definition of $E^{*}$ (see \eqref{D:lambdastar}).
\end{proof}

\begin{prop}
\label{P:liminfsigma}
 Let $\alpha\in(0,\pi)\cup(\pi,2\pi)$ and let $\bB\in \dS^2$ be a magnetic field which is not tangent to a face of the wedge $\cW_{\alpha}$. We have
\begin{equation}
\label{E:sinfEstar}
\sinf(\bB,\cS_{\alpha})= E^{*}(\bB,\cW_{\alpha}) \ .
\end{equation}
\end{prop}
\begin{rem}
The relation \eqref{E:sinfEstar} is not true when the magnetic field is tangent to a face of the wedge, see Proposition \ref{P:liminfTcase} and \eqref{E:lambdastarexplicit}.
\end{rem}
\begin{proof}
{\sc Lower bound:}
Let $(\chi_{1},\chi_{2})$ be two cut-off functions in $\mathcal{C}^{\infty}(\R_{+},[0,1])$ satisfying $\chi^2_{1}+\chi_{2}^2=1$, $\chi_{1}(r)=1$ if $r\in(0,\frac{1}{2})$ and $\chi_{1}(r)=0$ if $r\in(\frac{3}{4},+\infty)$. For $\tau\in \R^{*}$ we define the cut-off functions $\chi_{j,\tau}(x_{1},x_{2}):=\chi_{j}(\frac{r}{|\tau|})$ with $r=\sqrt{x_{1}^2+x_{2}^2}$. We have 
$$\exists C>0, \forall \tau\in \R^{*}, \quad \forall (x_{1},x_{2})\in \R^2, \quad \sum_{j=1}^{2}|\nabla\chi_{j,\tau}|^2 \leq \frac{C}{\tau^2} \ . $$
For $u\in \dom(\FQ_{\bB,\alpha}^{\tau})$, the IMS formula (see Lemma \ref{L:IMS}) provides 
\begin{equation}
\label{IMS:partitionrayon}
\FQ_{\bB,\alpha}^{\tau}(u) \geq \sum_{j=1}^{2}\FQ_{\bB,\alpha}^{\tau}(\chi_{j,\tau}u) - \frac{C}{\tau^2}\| u \|_{L^2}^2 \ . 
\end{equation}
Since $\supp(\chi_{1,\tau})\subset B(0,\frac{3}{4}\tau)$, we have $\dist(\Upsilon,\supp(\chi_{1,\tau}))\geq\frac{\tau}{4}$ and therefore we have $$\forall (x_{1},x_{2})\in \supp(\chi_{1,\tau}), \quad V_{\uB}^{\tau}(x_{1},x_{2})\geq\tfrac{1}{16}\tau^2 \ . $$ We deduce that for all $\tau\neq0$:
\begin{equation}
\label{E:fq1pol}
\FQ_{\bB,\alpha}^{\tau}(\chi_{1,\tau}u)\geq \tfrac{\tau^2}{16}\|\chi_{1,\tau}u\|^2_{L^2} \ .
\end{equation}
On the other part Lemma \ref{L:estinfty1} provides a constant $C_{1}>0$ such that for all $u\in \dom(\FQ_{\bB,\alpha}^{\tau})$ we have: 
$$ \forall \tau\in \R^{*}, \quad \FQ_{\bB,\alpha}^{\tau}(\chi_{2,\tau}u) \geq \left(E^{*}(\bB,\cW_{\alpha})-\frac{C_{1}}{\alpha^2\tau^2}\right)\| \chi_{2,\tau}u \|^2_{L^2}  \ . $$
We deduce by combining this with \eqref{IMS:partitionrayon} and \eqref{E:fq1pol} that 
$$\FQ_{\bB,\alpha}^{\tau}(u) \geq \min\left\{E^{*}(\bB,\cW_{\alpha})-\frac{C_{1}}{\alpha^2\tau^2},\frac{\tau^2}{16}\right\} \| u \|_{L^2}
^2-\frac{C}{\tau^2}\| u \|^2_{L^2} \ . $$
We deduce from the min-max principle that there exists $\tau_{0}>0$ such that for all $\tau$ satisfying $|\tau|>\tau_{0}$:
$$\sse(\bB,\cS_{\alpha};\tau) \geq E^{*}(\bB,\cW_{\alpha})-\frac{C_{1}}{\alpha^2\tau^2}-\frac{C}{\tau^2} \  $$
and therefore 
$$\sinf(\bB,\cS_{\alpha};\tau) \geq E^{*}(\bB,\cW_{\alpha}) \ . $$
\noindent {\sc Upper bound:} We suppose that $\theta^{+} \leq \theta^{-}$, the other case being symmetric. We have in that case $E^{*}(\bB,\cW_{\alpha})=\sigma(\theta^{+})$. Since we have assumed that we are not in the tangent case, we have $0<\theta^{+} $. Let $\epsilon>0$, we deduce from \eqref{E:infspecPup} that there exists $u_{\epsilon}\in \cC_{0}^{\infty}(\overline{\Hpup_{\alpha}})$ such that 
\begin{equation}
\label{E:energyqm}
\langle \left(H(\Add,\Hpup_{\alpha})+ V_{\uB}^{0}\right)u_{\epsilon},u_{\epsilon} \rangle_{L^2(\Hpup_{\alpha})}=\sigma(\theta^{+})+\epsilon \ . 
\end{equation}
We use $u_{\epsilon}$ to construct a quasimode of energy $\sigma(\theta^{+})+\epsilon$. Let $\bt^{+}:=(\cos\frac{\alpha}{2},\sin\frac{\alpha}{2})$ be a vector tangent to the boundary of $\Hpup_{\alpha}$. For $x=(x_{1},x_{2})$, we define the test-function: 
$$u_{\epsilon, \, \tau}(x):=e^{i \tau x\wedge\Add({\bt}^{+})}u_{\epsilon}(x-\tau \bt^{+}) \ .
$$
 We have $\supp(u_{\epsilon, \, \tau})=\supp(u_{\epsilon})+\tau \bt^{+}$. Since $\bt^{+}$ is pointing outward the corner of $\cS_{\alpha}$ along the upper boundary, there exists $\tau_{0}>0$ such that for all $\tau>\tau_{0}$ we have $\supp(u_{\epsilon, \, \tau}) \subset \overline{\cS_{\alpha}}$ and $\supp(u_{\epsilon,\, \tau})\cap \partial\Pi_{\alpha}^{-}=\emptyset$.  Therefore $u_{\epsilon,\, \tau}\in \dom(\FQ_{\bB,\alpha}^{\tau})$. Elementary computations (see the geometrical meaning of $V_{\uB}^{\tau}(x)$ in Subsection \ref{SS:essential}) provides $V_{\uB}^{\tau}(x-\tau\bt^{+})=V_{\uB}^{0}(x)$. Due to gauge invariance we get 
 $$\langle \left(H(\Add,\cS_{\alpha})+ V_{\uB}^{\tau}\right)u_{\epsilon,\tau},u_{\epsilon,\tau} \rangle_{L^2(\cS_{\alpha})}=\langle\left(H(\Add, \Hpup_{\alpha})+ V_{\uB}^{0}\right)u_{\epsilon},u_{\epsilon} \rangle_{L^2(\Hpup_{\alpha})}\ . $$
 We deduce from \eqref{E:energyqm} and from the min-max principle that 
 $$\forall \epsilon >0, \exists \tau_{0}>0, \forall \tau > \tau_{0}, \quad \sse(\bB,\cS_{\alpha};\tau) \leq \sigma(\theta^{+})+\epsilon \  $$
 and therefore $\liminf_{\tau\to+\infty}\sse(\bB,\cS_{\alpha};\tau) \leq \sigma(\theta^{+})$. Remark that in this proof we have taken $\tau \to +\infty$ in order to construct a test-function of energy close to $\sigma(\theta^{+})$. When $\theta^{-}\leq \theta^{+}$, the proof is the same but we use $\tau\to-\infty$.
 \end{proof}

\subsection{Comparison with the spectral quantities coming from the regular case}
\label{SS:Comparaison}
\begin{theo}
\label{T:comp}
Let $\alpha\in(0,\pi)\cup(\pi,2\pi)$ and $\bB\in \dS^2$, we have
\begin{equation}
\label{Up:complambdastar}
E(\bB,\cW_{\alpha}) \leq E^{*}(\bB,\cW_{\alpha}) \ .
\end{equation}
 Moreover if $E(\bB,\cW_{\alpha}) < E^{*}(\bB,\cW_{\alpha})$ then the band function $\tau\mapsto \sse(\bB,\cS_{\alpha};\tau)$ reaches its infimum. We denote by $\tau^{\rm c}\in \R$ a critical point such that 
 $$\sse(\bB,\cS_{\alpha};\tau^{\rm c})=E(\bB,\cW_{\alpha}) \ . $$
 Then there exists an eigenfunction with exponential decay for the operator $H(\Add,  \cS_{\alpha})+V_{\uB}^{\tau^{\rm c}}$ associated to the value $E(\bB,\cW_{\alpha})$.
 \end{theo}
 \begin{rem}
 Note that in the tangent case the band function $\tau\mapsto \sse(\bB,\cS_{\alpha};\tau)$ always reaches its infimum.
 \end{rem}
\begin{proof}
{\em Tangent case:} We have $E^{*}(\bB,\cW_{\alpha})=\Theta_{0}$ and \eqref{Up:complambdastar} is already proven (see \eqref{E:compcastangent}). Since the function $\tau\mapsto \sse(\bB,\cS_{\alpha};\tau)$ is continuous, we deduce from Proposition \ref{P:liminfTcase} and \eqref{E:relationfond} that the band function $\tau\mapsto \sse(\bB,\cS_{\alpha};\tau)$ reaches its infimum. Let $\tau^{\rm c}$ be a minimizer of $\sse(\bB,\cS_{\alpha};\tau)$. Assume that $E(\bB,\cW_{\alpha})<E^{*}(\bB,\cW_{\alpha})$. Since $\ssess(\bB,\cS_{\alpha};\tau^{\rm c})\geq \Theta_{0}$ (see Proposition \ref{P:specess}), $\sse(\bB,\cS_{\alpha};\tau^{\rm c})$ is a discrete eigenvalue of the operator $H(\Add,  \Secteur_{\alpha})+ V_{\uB}^{\tau^{\rm c}}$.

{\em Non tangent case:}  We deduce \eqref{Up:complambdastar} from Proposition \ref{P:liminfsigma} and \eqref{E:relationfond}. Assume that $E(\bB,\cW_{\alpha})<E^{*}(\bB,\cW_{\alpha})$. Since the function $\tau\mapsto \sse(\bB,\cS_{\alpha};\tau)$ is continuous, Proposition \ref{P:liminfsigma} and \eqref{E:relationfond} imply that the band function $\tau\mapsto \sse(\bB,\cS_{\alpha};\tau)$ reaches its infimum. We denote by $\tau^{\rm c}$ a Fourier parameter such that $E(\bB,\cW_{\alpha})=\sse(\bB,\cS_{\alpha};\tau^{\rm c})$. The bottom of the essential spectrum of $H(\Add,  \Secteur_{\alpha})+ V_{\uB}^{\tau^{\rm c}}$ is either $+\infty$ (ongoing case) or 1 (ingoing case), see Subsection \ref{SS:essential}. Since $E^{*}(\bB,\cW_{\alpha})<1$ we deduce that $E(\bB,\cW_{\alpha})$ is a discrete eigenvalue of the operator $H(\Add,  \Secteur_{\alpha})+ V_{\uB}^{\tau^{\rm c}}$. 

In both cases we denote by $u_{\tau^{\rm c}}$ an eigenfunction associated to $E(\bB,\cW_{\alpha})$ for the operator $H(\Add,  \Secteur_{\alpha})+ V_{\uB}^{\tau^{\rm c}}$. The fact that $u_{\tau^{\rm c}}$ has exponential decay is classical (see \cite{Ag82}) and we will give precise informations about the decay rate of the eigenfunctions in Proposition \ref{Agmunif1}. 
\end{proof}

Several particular cases where $E(\bB,\cW_{\alpha}) < E^{*}(\bB,\cW_{\alpha})$ can be found in literature (see \cite{Pan02}, \cite{Bon06} or \cite{Pof13T}). Theorem \ref{T:inegstricte} below gives geometrical conditions for this inequality to be satisfied. Let us also note that in \cite[Section 5]{Pof13T}, it is proved that $E(\bB,\cW_{\alpha}) =E^{*}(\bB,\cW_{\alpha})$ for a magnetic field tangent to a face, normal to the edge with an opening angle larger that $\frac{\pi}{2}$.

We now show that when \eqref{E:ineqwedge} is strict, there exists a generalized eigenfunction (in some sense we will define below) for $H(\bA,\cW_{\alpha})$ associated to the ground energy $E(\bB,\cW_{\alpha})$. This generalized eigenfunction is localized near the edge and can be used to construct quasimodes for the semiclassical magnetic Laplacian on a bounded domain with edges (see \cite{BoDauPof13}).

We denote by $L^2_{\rm loc}(\overline{\cW_{\alpha}})$ (respectively $H^1_{\rm loc}(\overline{\cW_{\alpha}}$)) the set of the functions $u$ which are in $L^2(\mathring{K})$ (respectively $H^1(\mathring{K})$) for all compact $K$ included in $\overline{\cW_{\alpha}}$ where $\mathring{K}$ denotes the interior of $K$.

We introduce the set of the functions which are {\em locally} in the domain of $H(\bA,\cW_{\alpha})$:
\begin{multline*}
   \dom_{\,\rm loc}\left(H(\bA,\cW_{\alpha})\right) :=
   \\ \{u\in H^1_{\rm loc}(\overline{\cW_{\alpha}}), \,  (-i\nabla-\bA)^2u\in L^2_{\rm loc}(\overline{\cW_{\alpha}}), \, 
  (-i\nabla-\bA)u\cdot {\bf n}=0 \ \ \rm{ on } \ \ \partial\cW_{\alpha} \}\, ,
\end{multline*}
where ${\bf n}$ is the outward normal of the boundary $\partial\cW_{\alpha}$ of the wedge.
\begin{cor}
\label{C:generalizedef}
Let $\alpha\in(0,\pi)\cup(\pi,2\pi)$ and $\bB\in \dS^2$. Assume $E(\bB,\cW_{\alpha}) < E^{*}(\bB,\cW_{\alpha})$. Then there exists a non-zero function $\psi\in \dom_{\,\rm loc}\left(H(\bA,\cW_{\alpha})\right) $ satisfying 
\begin{equation*}
\left \{
\begin{aligned}
&(-i\nabla-\bA)^2\psi=E(\bB,\cW_{\alpha})\psi \quad \mbox{in}\quad  \cW_{\alpha}
\\
&(-i\nabla-\bA)\psi\cdot {\bf n}=0 \quad \mbox{on}\quad  \partial\cW_{\alpha} \ .
\end{aligned}
\right.
\end{equation*}
 Moreover $\psi$ has exponential decay in the $(x_{1},x_{2})$ variables.
\end{cor}
\begin{proof}
Let $\tau^{\rm c}$ be a minimizer of $\tau\mapsto \sse(\bB,\cS_{\alpha};\tau)$ given by Theorem \ref{T:comp}. Let $\Phi_{\tau^{\rm c}}$ be an eigenfunction of $H(\Add, \cS_{\alpha})+V_{\uB}^{\tau^{\rm c}}$ associated to $E(\bB,\cW_{\alpha})$. It has exponential decay and satisfies the boundary condition $(-i\nabla-\Add)u_{\tau^{\rm c}}\cdot \underline{{\bf n}}=0$ where $\underline{{\bf n}}$ is the outward normal to the boundary of $\cS_{\alpha}$. Let 
\begin{equation}
\label{E:generalizedadm}
\psi(x_{1},x_{2},x_{3}):=e^{i\tau^{\rm c}x_{3}}\Phi_{\tau^{\rm c}}(x_{1},x_{2}) \ . 
\end{equation}
We clearly have $u\in \dom_{\rm loc}(H(\bA,\cW_{\alpha}))$. Moreover writing $\bA=(\Add,x_{2}b_{1}-x_{1}b_{2})$ we get $$(-i\nabla-\bA)^2\psi=\left((-i\nabla_{x_{1},x_{2}}-\Add)^2\Phi_{\tau^{\rm c}}+(\tau^{\rm c}-x_{2}b_{1}+x_{1}b_{2})^2\Phi_{\tau^{\rm c}}\right)e^{i\tau^{\rm c}x_{3}}=E(\bB,\cW_{\alpha})\psi \ . $$
Therefore $\psi$ satisfies the conditions of the corollary.
\end{proof}
We say that the function $\psi$ is a generalized eigenfunction of $H(\bA,\cW_{\alpha})$. Since it has the form \eqref{E:generalizedadm}, we say it is \em admissible \rm and we shall use it to construct quasimode for the operator $(-ih\nabla-\bA)^2$ on $\Omega$ when $\Omega$ has an edge (see \cite{BoDauPof13}).

%
%
%

\section{Continuity}
\label{S:continuity}
In this section we prove the continuity of the application $(\bB,\alpha)\mapsto E(\bB,\cW_{\alpha})$. The domain of the quadratic form $\FQ_{\bB,\alpha}^{\tau}$ depends on the geometry (see \eqref{D:DomFQ}), moreover the bottom of the spectrum of the operator $H(\Add,  \cS_{\alpha})+V_{\uB}^{\tau}$ may be essential, see Subsection \ref{SS:essential}. Therefore we cannot apply directly Kato's perturbation theory.

In this section we use the generic notation $g$ (like geometry) for a couple $(\bB,\alpha)\in \dS^2\times (0,2\pi)$. We denote by $E(g):=E(\bB,\cW_{\alpha})$ and $\sse(g;\tau):=\sse(\bB,\cS_{\alpha};\tau)$. We also note $\FQ_{g}^{\tau}$ the quadratic form $\FQ_{\bB,\alpha}^{\tau}$ 
\subsection{Uniform Agmon estimates}
Here we give Agmon's estimates of concentration for the eigenfunctions of the operator $H(\Add,  \cS_{\alpha})+V_{\uB}^{\tau}$ associated to the ground energy $E(g)$. 

First we recall a basic commutator formula (see \cite[Chapter 3]{CyFrKiSi87}):
\begin{lem}
 Let $\Phi$ be a uniformly Lipschitz function on $\cS_{\alpha}$ and let $(E,u)$ be an eigenpair of the operator $H(\Add,  \cS_{\alpha})+V_{\uB}^{\tau}$. Then we have 
 \begin{equation}
 \label{IMS22}
\forall u\in \dom(\FQ_{g}^{\tau}), \quad \FQ_{g}^{\tau}(e^{\Phi} u)=\int_{\cS_{\alpha}}\left(E+|\nabla \Phi|^2 \right) e^{2\Phi} |u|^2 \ . 
 \end{equation}
\end{lem}

We introduce the lowest energy of $H(\Add,  \cS_{\alpha})+V_{\uB}^{\tau}$ far from the origin: 
\begin{equation}
\label{D:Estar}
\cE^{*}(g):=
\left\{ 
\begin{aligned}
&E^{*}(\bB,\cW_{\alpha}) \quad \mbox{if}\quad \alpha\neq\pi \ , 
\\
&E(\bB,\cW_{\alpha}) \quad \mbox{if} \quad \alpha=\pi \ .
\end{aligned}
\right.
\end{equation}
We have $\cE^{*}(g)=\sigma(\theta^{0})$ where $\theta^{0}$ is the minimum angle between the magnetic field and the boundary of $\cW_{\alpha}$. Since $\theta\mapsto\sigma(\theta)$ is continuous we deduce that $g\mapsto \cE^{*}(g)$ is continuous on $\dS^2\times (0,2\pi)$. 

The following proposition gives the exponential decay uniformly with respect to the geometry for the first eigenfunctions of $H(\Add,  \cS_{\alpha})+V_{\uB}^{\tau(g)}$, provided that there exists a gap between $E(g)$ and $\cE^{*}(g)$:
\begin{prop}
\label{Agmunif1}
Let $G\subset \dS^2\times ((0,2\pi)\setminus \pi)$. We suppose that there exist $\alpha_{0}>0$ and $\delta\in(0,1)$ such that for all $g=(\bB,\alpha)\in G$ we have $\alpha \geq \alpha_{0}$ and $\cE^{*}(g)-E(g) \geq \delta$. Let $\tau(g)\in \R$ be a value of the Fourier parameter given in Theorem \ref{T:comp} such that $\sse(g;\tau(g))=E(g)$. Let $\phi_{\nu}(x_{1},x_{2}):=\nu\sqrt{x_{1}^2+x_{2}^2}$ be an Agmon distance. Then for all $\nu \in (0,\sqrt{\delta})$ there exists $C(\nu)>0$ such that for all $g\in G$ and for all eigenfunction $u_{g}$ of $H(\Add,  \cS_{\alpha})+V_{\uB}^{\tau(g)}$ associated to $\sse(g;\tau(g))$ we have  
$$\FQ_{g}^{\tau(g)}(e^{\phi_{\nu}}u_{g}) \leq C(\nu) \| u_{g}\|^2_{L^2} \ . $$
\end{prop}
\begin{proof}
 We know from the results of \cite{Ag82} that $e^{\phi_{\nu}}u_{g}\in L^2(\cS_{\alpha})$. Since $|\nabla \phi_{\nu}|^2=\nu^2$ the IMS formula \eqref{IMS22} provides
\begin{equation}
\label{IMS:qm}
\int_{\cS_{\alpha}}(\sse(g;\tau(g))+\nu^2)e^{2\phi_{\nu}}|u_{g}|^2=\FQ_{g}^{\tau(g)}(e^{\phi_{\nu}}u_{g}) \ . 
\end{equation}
We use cut-off functions $\chi_{1,R}$ and $\chi_{2,R}$ in $C^{\infty}(\cS_{\alpha},[0,1])$ that satisfy $\chi_{1,R}(x)=0$ when $|x| \geq 2R$ and $\chi_{1,R}(x)=1$ when $|x| \leq R$ and $\chi_{1,R}^2+\chi_{2,R}^2=1$. We also assume without restriction that there exists $C>0$ such that
\begin{equation}
\label{E:controletaderivee}
\forall R>0, \quad \sum_{j=1}^{2}|\nabla \chi_{j,R}|^2 \leq \frac{C}{R^2} \ . 
\end{equation}
Lemma \ref{L:IMS} provides 
$$\FQ_{g}^{\tau(g)}(e^{\phi_{\nu}}u_{g}) = \sum_{j=1}^{2} \FQ_{g}^{\tau(g)}(\chi_{j,R}e^{\phi_{\nu}}u_{g})-\sum_{j=1}^{2}\| \nabla\chi_{j,R}e^{\phi_{\nu}}u_{g}\|^2  \ $$
and from \eqref{IMS:qm} and \eqref{E:controletaderivee} we get
$$ \left(\sse(g;\tau(g))+\nu^2+\frac{C}{R^2} \right) \| e^{\phi_{\nu}}u_{g}\|^2_{L^2} \geq \sum_{j=1}^{2}\FQ_{g}^{\tau(g)}(\chi_{j,R}e^{\phi_{\nu}}u_{g}) \ . $$
When $\alpha\neq\pi$ Lemma \ref{L:estinfty1} provides $\FQ_{g}^{\tau(g)}(\chi_{2,R}u_{g}) \geq (\cE^{*}(g)-\frac{C_{1}}{\alpha^2 R^2})\| \chi_{2,R}e^{\phi_{\nu}}u_{g}\|^2_{L^2}$, therefore we have for all $g\in G$:
\begin{multline}
\label{E:inegAgmonclef}
\left(\cE^{*}(g)-\sse(g;\tau(g))-\nu^2-\frac{C_{1}}{R\alpha_{0}^2}-\frac{C}{R^2} \right)\| \chi_{2,R}e^{\phi_{\nu}}u_{g}\|^2_{L^2} 
\\
\leq \left(\sse(g;\tau(g))+\nu^2+\frac{C}{R^2}+\frac{C_{1}}{R\alpha_{0}^2} \right) \| \chi_{1,R}e^{\phi_{\nu}}u_{g}\|^2_{L^2} \ .
\end{multline}
From the hypotheses on $G$ we can choose $R>0$ and $\epsilon>0$ such that: 
$$\forall g\in G, \quad \epsilon \leq \left(\cE^{*}(g)-\sse(g;\tau(g))-\nu^2-\frac{C_{1}}{R^2\alpha_{0}^2}-\frac{C}{R^2} \right) \ . $$
Moreover since $\sse(g;\tau(g))<1$ and $\nu < \sqrt{\delta}\leq1$ we can choose $R$ such that $\sse(g;\tau(g))+\nu^2+\frac{C}{R^2} +\frac{C_{1}}{R^2\alpha_{0}^2}\leq 2$. We deduce from \eqref{E:inegAgmonclef}
$$\forall g\in G, \quad \| e^{\phi_{\nu}}u_{g}\|^2_{L^2} \leq \left(\frac{2}{\epsilon}+1\right)\| \chi_{1,R}e^{\phi_{\nu}}u_{g}\|^2_{L^2} \leq \left(\frac{2}{\epsilon}+1\right)e^{2\nu R}\|u_{g}\|_{L^2}^2 \ . $$
We deduced the estimate on the quadratic form from the IMS formula \eqref{IMS:qm}.
\end{proof}

\subsection{Polar coordinates}
Let $(r,\psi)\in \R_{+}\times (-\frac{\alpha}{2},\frac{\alpha}{2})$ be the usual polar coordinates of $\cS_{\alpha}$. We use the change of variables associated to the normalized polar coordinates $(r,\phi):=(r,\frac{\psi}{\alpha})\in\Omega_{0}:=\R_{+}\times(-\frac{1}{2},\frac{1}{2})$. After a change a gauge (see \cite[Section 3]{Bon06} and \cite[Section 5]{Pof13T}) we get that the quadratic form $\FQ_{g}^{\tau}$ is unitary equivalent to the quadratic form
\begin{equation}
\label{D:FQpol}
\FQpol_{g}^{\tau}(u):=\int_{\Omega_{0}}\left(|(\partial_{r}-i\alpha r \phi b_{3})u|^2+\frac{1}{\alpha^2 r^2}|\partial_{\phi}u|^2+\Vpol_{g}^{\tau}(r,\phi)|u|^2\right)r \d r \d \phi 
\end{equation}
  with the electric potential in polar coordinates:
    \begin{equation}
\label{}
\Vpol_{g}^{\tau}(r,\phi):=\big(r\cos(\phi\alpha)b_{2}-r\sin(\phi\alpha)b_{1}-\tau\big)^2 \ .
\end{equation} 
The form domain is 
\begin{equation*}
\label{Domaine_fqpol}
\dom(\FQpol_{g}^{\tau})=
\left\{ u\in L^2_{r}(\Omega_{0}), \, (\partial_{r}-i\alpha r \phi b_{3})u\in L^2_{r}(\Omega_{0}),\,  \frac{1}{r}\partial_{\phi}u\in L^2_{r}(\Omega_{0}),\,  \sqrt{\Vpol_{g}^{\tau}}\, u\in L^2_{r}(\Omega_{0}) \right\}
\end{equation*}
where $L^2_{r}(\Omega_{0})$ stands for the set of the square-integrable functions for the weight $r\d r$.

\begin{notation}
Let $g_{0}=(\bB_{0},\alpha_{0})\in\dS^{2}\times (0,2\pi)$ and $\eta>0$. We denote by $B(g_{0},\eta)$ the ball of $\dS^{2}\times \R$ of center $g_{0}$ and of radius $\eta$ related to the euclidean norm $\|g\|:=\left(\|\bB\|^2_{2}+\alpha^2\right)^{1/2}$.
\end{notation}


\begin{lem}
\label{L:pertpolaire}
Let $g_{0}=(\bB_{0},\alpha_{0})\in\dS^{2}\times (0,2\pi)$. There exist $C>0$ and $\eta>0$ such that $B(g_{0},\eta)\subset \dS^{2}\times (0,2\pi)$ and for all $g\in B(g_{0},\eta)$ we have for all
$u\in\dom(\FQpol_{g}^{\tau})\cap \dom(\FQpol_{g_{0}}^{\tau})$:
$$\forall \tau\in \R, \quad \FQpol_{g}^{\tau}(u) \leq \FQpol_{g_{0}}^{\tau}(u) + C \|g-g_{0}\| \left(\| ru \|_{L^2_{r}(\Omega_{0})} ^2 +\FQpol_{g_{0}}^{\tau}(u)\right) \ .$$
\end{lem}
\begin{proof}
Let $g_{0}=(\bB_{0},\alpha_{0})$ and $g=(\bB,\alpha)$ be in $\dS^2\times(0,2\pi)$. We denote by $(b_{j,0})_{j}$ and $(b_{j})_{j}$ the cartesian coordinates of $\bB_{0}$ and $\bB$. Let $d:=\|g-g_{0}\|$. We discuss separately the three terms of $\FQpol_{g}^{\tau}(u)$ written in \eqref{D:FQpol}. For the first one we write 
\begin{align*}
|(\partial_{r}-i\alpha b_{3}r \phi)u|^2 \leq & |(\partial_{r}-i\alpha_{0} b_{3,0}r \phi)u|^2
\\
&+|\alpha_{0}b_{3,0}-\alpha b_{3}|^2 r^2|u|^2+2r|u| |\alpha_{0}b_{3,0}-\alpha b_{3}||(\partial_{r}-i\alpha_{0} b_{3,0}r \phi)u| 
\end{align*}
We have $|\alpha_{0}b_{3,0}-\alpha b_{3}|\leq \alpha_{0}|b_{3,0}-b_{3}|+|b_{3}||\alpha_{0}-\alpha|\leq (2\pi+1)d,$ and $$2r|u| |\alpha_{0}b_{3,0}-\alpha b_{3}||(\partial_{r}-i\alpha_{0} b_{3,0}r \phi)u|\leq (2\pi+1)d\left(r^2|u|^2+|(\partial_{r}-i\alpha_{0} b_{3,0}r \phi)u|^2\right) \ .$$
Therefore there exists $C_{1}>0$ such that for all  $g\in \dS^2\times (0,2\pi)$:
\begin{multline}
\label{E:Pert1}
|(\partial_{r}-i\alpha b_{3}r \phi)u|^2 \leq  |(\partial_{r}-i\alpha_{0} b_{3,0}r \phi)u|^2
\\
+C_{1}d\left( r^2|u|^2(1+d)+|(\partial_{r}-i\alpha_{0} b_{3,0}r \phi)u|^2\right) \ .
\end{multline}
We deal with the second term: we have
\begin{align*}\left|\frac{1}{\alpha^2r^2}|\partial_{\phi}u|^2-\frac{1}{\alpha_{0}^2r^2}|\partial_{\phi}u|^2\right| 
= d\frac{\alpha+\alpha_{0}}{\alpha^2\alpha_{0}}\frac{1}{\alpha_{0}r^2}|\partial_{\phi}u|^2   \ .
\end{align*}
Therefore there exist $\eta>0$ and $C_{2}>0$ such that $B(g_{0},\eta)\subset \dS^2\times (0,2\pi)$ and
\begin{equation}
\label{E:Pert2}
\forall g\in B(g_{0},\eta), \quad \frac{1}{\alpha^2r^2}|\partial_{\phi}u|^2\leq\frac{1}{\alpha_{0}^2r^2}|\partial_{\phi}u|^2
+ C_{2}d\frac{1}{\alpha_{0}r^2}|\partial_{\phi}u|^2   \ .
\end{equation}
For the third term we write 
\begin{align*}
\Vpol_{g}^{\tau}(r,\phi)\leq &\Vpol_{g_{0}}^{\tau}(r,\phi) + \left|\cos(\alpha\phi)b_{2}-\cos(\alpha_{0}\phi)b_{2,0}+\sin(\alpha_{0}\phi)b_{1,0}-\sin(\alpha\phi)b_{1} \right|^2r^2
\\
&+2\sqrt{\Vpol_{g_{0}}^{\tau}(r,\phi)}  \left|\cos(\alpha\phi)b_{2}-\cos(\alpha_{0}\phi)b_{2,0}+\sin(\alpha_{0}\phi)b_{1,0}-\sin(\alpha\phi)b_{1} \right| r \ .
\end{align*}
We get $C_{3}>0$ and $C_{4}>0$ such that for all $g\in \dS^2\times (0,2\pi)$ and for all $\tau\in \R$:
\begin{equation}
\label{E:Pert3}
\forall (r,\phi)\in \Omega_{0}, \quad \Vpol_{g, \tau}(r,\phi)\leq (1+C_{3}d)\Vpol_{g_{0}}^{\tau}(r,\phi)+C_{4}r^2 d \ .
\end{equation}
Combining \eqref{E:Pert1}, \eqref{E:Pert2} and \eqref{E:Pert3} we get $C>0$ such that for all $g\in B(g_{0},\eta)$:
$$\FQpol_{g}^{\tau}(u) \leq \FQpol_{g_{0}}^{\tau}(u)| + C \| g-g_{0}\| \left(\| ru \|_{L^2_{r}(\Omega_{0})} ^2 +\FQpol_{g_{0}}^{\tau}(u)\right) \ .$$
\end{proof}

\subsection{Main result}
\begin{theo}
\label{P:continuitediedre}
The function $g\mapsto E(g)$ is continuous on $\dS^{2}\times(0,2\pi)$. 
\end{theo}
\begin{proof}
Let $g_{0}\in \dS^2\times (0,2\pi)$. We distinguish different cases depending on whether \eqref{Up:complambdastar} is strict or not.

\noindent {\em Case 1}: When 
\begin{equation}
\label{H:Incas1}
E(g_{0}) < \cE^{*}(g_{0}) \ . 
\end{equation}
Let us note that in that case $\alpha_{0}\neq\pi$ (see \eqref{D:Estar}). We use Theorem \ref{T:comp}: There exists $\tau^{\rm c}\in \R$ such that the band function $\tau \mapsto \sse(g_{0};\tau)$ reaches its infimum in $\tau^{\rm c}$ and there exists a normalized eigenfunction (in polar coordinate) $u_{0}$ for $\FQpol_{g_{0}}^{\tau^{\rm c}}$ with exponential decay in $r$. We use this function as a quasimode for $\FQpol_{g}^{\tau^{\rm c}}$. We get from Lemma \ref{L:pertpolaire} constants $C>0$ and $\eta>0$ such that  for all $g\in B(g_{0},\eta)$:
\begin{align*}
\FQpol_{g}^{\tau^{\rm c}}(u_{0}) &\leq 
\FQpol_{g_{0}}^{\tau^{\rm c}}(u_{0})+C\|g-g_{0}\|\left(\|ru_{0}\|_{L^2_{r}}^2+\FQpol_{g_{0}}^{\tau^{\rm c}}(u_{0})\right)
\\
&=E(g_{0})+C\|g-g_{0}\|\left(\|ru_{0}\|_{L^2_{r}}^2+E(g_{0})\right) \ . 
\end{align*}
Let $\epsilon>0$. Since $u_{0}$ has exponential decay in $r$ we get for $g$ close enough to $g_{0}$:
$$ \FQpol_{g}^{\tau^{\rm c}}(u_{0}) \leq E(g_{0})+\epsilon \ . $$
The min-max Principle and the relation \eqref{E:relationfond} provide
\begin{equation}
\label{E:contlimsup}
 \limsup_{g\to g_{0}}E(g) \leq E(g_{0})  \ . 
 \end{equation}
Using this upper bound, the assumption \eqref{H:Incas1} and the continuity of $g\mapsto \cE^{*}(g)$, we deduce that there exist $\kappa>0$ and $\epsilon_{0}>0$ such that $\overline{B(g_{0},\kappa)}\subset \dS^2\times((0,2\pi)\setminus \pi)$ and
  \begin{equation}
  \label{E:distspecess}
  \forall g\in B(g_{0},\kappa), \quad \epsilon_{0} < \cE^{*}(g)-E(g) \ .
  \end{equation}
   Let  $g\in B(g_{0},\kappa)$, Theorem \ref{T:comp} provides  $\tau(g)\in \R$ such that $\sse(g;\tau(g))=E(g)$ is a discrete eigenvalue for the operator $H(\Add,\cS_{\alpha})+V_{\uB}^{\tau(g)}$. We denote by $u_{g}$ an associated normalized eigenfunction in polar coordinates.
We apply the uniform exponential estimates of Proposition \ref{Agmunif1} to the set $G:=B(g_{0},\kappa)$ and we get:
  \begin{equation}
\label{E:concentrationexpopreuveC}
 \forall \nu\in(0,\sqrt{\epsilon_{0}}), \ \exists C_{0}(\nu), \ \forall g\in B(g_{0},\kappa),
\quad \| e^{\nu \,r} u_{g} \|_{L^2_{r}(\Omega_{0})}<C_{0}(\nu) \ .
\end{equation}
We use $u_{g}$ as a quasimode for $\FQpol_{g_{0}}^{\tau(g)}$: \eqref{E:concentrationexpopreuveC} and Lemma \ref{L:pertpolaire} yields
  $$\exists C_{1}>0,\ \forall g\in B(g_{0},\kappa), \quad \FQpol_{g_{0}}^{\tau(g)}(u_{g}) \leq \FQpol_{g}^{\tau(g)}(u_{g})+C_{1}\| g-g_{0}\|$$
 and since $u_{g}$ satisfies $\FQpol_{g}^{\tau(g)}(u_{g})=E(g)$ we deduce from the min-max principle and \eqref{E:relationfond}:
 $$\forall g\in B(g_{0}, \kappa), \quad E(g_{0}) \leq E(g)+C_{1}\|g-g_{0} \|  \ .$$ 
This last upper bound combined with \eqref{E:contlimsup} brings the continuity of $E(\cdot)$ in $g_{0}$ when $E(g_{0})<\cE^{*}(g_{0})$.
 
\noindent {\em Case 2}: When
\begin{equation}
\label{H:Incas2}
E(g_{0}) = \cE^{*}(g_{0}) \ . 
\end{equation}
Let us suppose that for all $\epsilon>0$ there exists $\kappa>0$ such that for all $g\in B(g_{0},\kappa)$  we have 
$$\cE^{*}(g)-\epsilon \leq E(g) \leq \cE^{*}(g) \ . $$ In that case we deduce the continuity of $E(\cdot)$ in $g_{0}$ from the continuity of $\cE^{*}(\cdot)$.

Let us write the contraposition of the previous statement and exhibit a contradiction. We suppose that there exists $\epsilon_{0}>0$ such that for all $\kappa>0$ there exists $g\in \dS^2\times(0,2\pi)$ satisfying $\|g-g_{0}\|< \kappa$ and $E(g) < \cE^{*}(g)-\epsilon_{0}$. This implies $\alpha\neq\pi$ (see \eqref{D:Estar}). Theorem \ref{T:comp} provides $\tau(g)\in \R$ such that $E(g)=\sse(g;\tau(g))$ and we denote by $u_{g}$ an associated normalized eigenfunction for $H(\Add,\cS_{\alpha})+V_{\uB}^{\tau(g)}$. Again Proposition \ref{Agmunif1} shows that this eigenfunction has exponential decay uniformly in $g$: For each $\nu\in (0,\sqrt{\epsilon_{0}})$, we have $C_{0}>0$ that does not depend on $g$ such that
$$\|e^{\nu \, r} u_{g} \|_{L^2_{r}(\Omega_{0})}<C_{0} \  .$$
We use $u_{g}$ as a quasimode for $\FQpol_{g_{0}}^{\tau(g)}$: There exists a constant $C_{1}>0$ that does not depend on $g$ such that
\begin{align*}
 \FQpol_{g_{0}}^{\tau(g)}(u_{g}) &\leq \FQpol_{g}^{\tau(g)}(u_{g})+C_{1}\|g-g_{0}\| 
\\
&<\cE^{*}(g)-\epsilon_{0}+C_{1}\kappa. 
\end{align*}
The min-max Principle and \eqref{E:relationfond} provide $$E(g_{0}) < \cE^{*}(g)-\epsilon_{0}+C_{1}\kappa \ . $$ Let $\epsilon>0$, the continuity of $\cE^{*}(\cdot)$ implies that for $\kappa>0$ small enough there holds $\cE^{*}(g)<\cE^{*}(g_{0})+\epsilon$. We have proved:
$$\exists\epsilon_{0}>0,\,  \exists C_{1}>0,\,  \forall \epsilon>0, \, \exists \kappa_{0}>0, \forall \kappa\in(0,\kappa_{0}),\quad E(g_{0}) < \cE^{*}(g_{0})-\epsilon_{0}+C_{1}\kappa+\epsilon \ . $$
Choosing $\epsilon>0$ and $\kappa>0$ small enough we get a contradiction with \eqref{H:Incas2}.
\end{proof}

\section{Upper bound for small angles}
\label{S:UbSa}
\subsection{An auxiliary problem on a half-line}
\label{S:HL}
Let $L_{r}^{2}(\R_{+})$ be the space of the square-integrable functions for the weight $r\d r$ and let $$B^1_{r}(\Rp) :=\{ u\in L^2_{r}(\R^+), u'\in L^2_{r}(\R^+), r u \in L^2_{r}(\R^+)\} \ .$$ 
We define the 1d quadratic form
$$ \gq_{\tau}(u):=\int_{\R_{+}}\left(|u'(r)|^2+(r-\tau)^2|u(r)|^2\right) r \d r $$
on the domain $B_{r}^{1}(\R_{+})$. 
As we will see later, if $u$ is a function of $L^2(\cS_{\alpha})$ that does not depend on the angular variable and if $b_{2}\neq 0$, $b_{2}^{-1}\FQ_{\bB,\alpha}^{\tau}(u)$ written in polar coordinates degenerates formally toward $\gq_{\tau}(u)$ when $\alpha$ goes to 0. 

We denote by $\gg_{\tau}$ the Friedrichs extension of the quadratic form $\gq_{\tau}$. This operator has been introduced in \cite{Yaf08} and studied in \cite{Pof13-II} as the reduced operator of a 3d magnetic hamiltonian with axisymmetric potential. 

The technics from \cite{BolCa72} show that $\gg_{\tau}$ has compact resolvent. We denote by $\zeta(\tau)$ its first eigenvalue. For all $\tau\in \R$, $\zeta(\tau)$ is a simple eigenvalue and we denote by $z_{\tau}$ an associated normalized eigenfunction. Basic estimates of Agmon show that $z_{\tau}$ has exponential decay. The following properties are shown in \cite{Pof13-II}:

 The function $\tau\mapsto\zeta(\tau)$ reaches its infimum. We denote by 
\begin{equation}
\label{D:xi0}
\Xi_{0}:=\inf_{\tau\in\R}\zeta(\tau) 
\end{equation}
 the infimum. Let $\tau_{0}>0$ be the lowest real  number such that $\zeta(\tau_{0})=\Xi_{0}$. We have 
 \begin{equation}
\label{E:inexi0}
 \Theta_{0}<\Xi_{0}\leq \sqrt{4-\pi} \ . 
 \end{equation}
 Numerical simulations show that $\Xi_{0}\approx 0.8630$.

\subsection{Upper bounds and consequences}
\label{SS:upperboundsmallangles}
Let $\bB=(b_{1},b_{2},b_{3})$ be a magnetic field in $\dS^2$. Due to symmetry we assume $b_{j}\geq0$ for all $j\in\{1,2,3\}$ (see Proposition \ref{P:symmetry}). Recall the quadratic form $\FQpol_{\bB,\alpha}^{\tau}$ associated to $H(\Add, \cS_{\alpha})+V_{\uB}^{\tau}$ in polar coordinates (see \eqref{D:FQpol}). The injection from $\left(B_{r}^{1}(\R_{+}),\|\cdot\|_{L^2_{r}(\R_{+})}\right)$ into $\left(\dom(\FQpol_{{\bf B},\alpha}^{\tau}(u)),\|\cdot\|_{L^2_{r}(\Omega_{0})}\right)$ is an isometry, therefore we can restrict $\FQpol_{{\bf B},\alpha}^{\tau}$ to $B_{r}^{1}(\R_{+})$ and in the following for $u\in B^1_{r}(\R_{+})$ we denote again by $u$ the associated function defined on $\Omega_{0}$. Assume $b_{2}>0$, that means that the magnetic field is not tangent to the symmetry plane of the wedge. For $u\in B^1_{r}(\R_{+})$ we have formally that $b_{2}^{-1}\FQpol_{\bB,\alpha}^{ \tau\sqrt{b_{2}}}(u)$ goes to $\gq_{\tau}(u)$ when $\alpha$ goes to 0. 

The following lemma makes this argument more rigorous: 
\begin{lem}
\label{L:evalqm}
Let $\bB\in \dS^2$ with $b_{2}>0$. For $u\in B_{r}^{1}(\R_{+})$ we denote by $u^{\rm sc}(r):=b_{2}^{1/2}u(b_{2}^{1/2}r)$ the associated rescaled function. We have $\|u^{\rm sc}\|_{L^2_{r}(\R_{+})}=\|u\|_{L^2_{r}(\R_{+})}$ and 
\begin{multline}
\label{E:evalqm}
\FQpol_{{\bf B},\alpha}^{\tau\sqrt{b_{2}}}(u^{\rm sc})=b_{2}\gq_{\tau}(u)+\frac{\alpha^2}{12}\| r u\|^2_{L^2_{r}(\Rp)}\frac{b_{3}^2}{b_{2}}
\\
+\frac{1}{2}(1-\sinc \alpha)\|r u\|^2_{L^2_{r}(\Rp)}\frac{b_{1}^2-b_{2}^2}{b_{2}}+2\tau b_{2}\left(1-\sinc \frac{\alpha}{2}\right)\|\sqrt{r} u\|^2_{L^2_{r}(\Rp)} \ .
\end{multline}
with $\sinc\alpha:=\frac{\sin\alpha}{\alpha}$.
\end{lem}
\begin{proof}
 We evaluate $\FQpol_{{\bf B},\alpha}^{\tau}(u)$ for $u\in B_{r}^{1}(\Rp)$:
 \begin{multline*}
 \FQpol_{{\bf B},\alpha}^{\tau}(u)=\int_{\Rp}\left(|u'(r)|^2+(rb_{2}-\tau)|u(r)|^2\right)r \d r
 \\+\int_{\Omega_{0}}\alpha^2\phi^2 b_{3}^2r^2|u(r)|^2 r\d r\d\phi+\int_{\Omega_{0}}\left(\Vpol_{\uB}^{\tau}(r,\phi)-(rb_{2}-\tau)^2\right)|u(r)|^2 r\d r \d\phi \ .
 \end{multline*}
 We have
\begin{equation}
\label{Eval_QM1}
\int_{\Omega_{0}}\alpha^2 r^2 \phi^2 b_{3}^2 |u(r)|^2 r\d r \d\phi=\frac{\alpha^2}{12}\| r u\|_{L^2_{r}(\Rp)}^2b_{3}^2 \ . 
\end{equation}
Elementary computations yield: 
 \begin{align*}
 \Vpol_{\uB}^{\tau}(r,\phi)-(rb_{2}-\tau)^2&=r^2 \sin^2(\alpha\phi)(b_{1}^2-b_{2}^2)-2rb_{2}\tau\left(\cos(\alpha\phi)-1\right)^2
 \\
 & \quad - 2r b_{1}\sin(\alpha\phi)\left(r b_{2}\cos(\alpha\phi)-\tau\right) \ .
 \end{align*}
 Since the term $- 2r b_{1}\sin(\alpha\phi)\left(r b_{2}\cos(\alpha\phi)-\tau\right)$ is odd with respect to $\phi$, its integral on $\Omega_{0}$ vanishes. For the other terms we use:
 \begin{equation*}
 \label{inteta1}
 \int_{-1/2}^{1/2}\sin^2(\alpha\phi)\d\phi=\tfrac{1}{2}(1-\sinc\alpha) \quad \mbox{and}\quad \int_{-1/2}^{1/2}\big(\cos(\alpha\phi)-1\big)\d\phi=\sinc\tfrac{\alpha}{2}-1 \ .
\end{equation*}
We deduce for all $u\in B^1_{r}(\Rp)$ and $\tau\in \R$:
\begin{multline*}
\int_{\Omega_{0}}\left(\Vpol_{\uB}^{\tau}(r,\phi)-(rb_{2}-\tau)^2\right)|u(r)|^2 r\d r \d\phi =
\\
\tfrac{1}{2}(1-\sinc \alpha)\|r u\|^2_{L^2_{r}(\Rp)}(b_{1}^2-b_{2}^2)+2\tau\left(1-\sinc \tfrac{\alpha}{2}\right)\|\sqrt{r} u\|^2_{L^2_{r}(\Rp)}b_{2} \ 
\end{multline*}
and therefore (note that we have make the change $\tau\to\tau\sqrt{b_{2}}$):
\begin{multline}
\label{E:evalqmlem}
\FQpol_{{\bf B},\alpha}^{\tau\sqrt{b_{2}}}(u)=\int_{\Rp}\left(|u'(r)|^2+(rb_{2}-\tau\sqrt{b_{2}})^2|u(r)|^2\right)r\d r+\tfrac{\alpha^2}{12}\| r u\|^2_{L^2_{r}(\Rp)}b_{3}^2
\\
+\tfrac{1}{2}(1-\sinc \alpha)\|r u\|^2_{L^2_{r}(\Rp)}(b_{1}^2-b_{2}^2)+2\tau\left(1-\sinc \tfrac{\alpha}{2}\right)\|\sqrt{r} u\|^2_{L^2_{r}(\Rp)}b_{2}^{3/2} \ .
\end{multline}
Let $u^{\rm sc}(r):=b_{2}^{1/2}u(b_{2}^{1/2}r)$. An elementary scaling provides 
$$\int_{\Rp}\left(|(u^{\rm sc})'(r)|^2+(rb_{2}-\tau\sqrt{b_{2}})^2|u^{\rm sc}(r)|^2\right)r\d r=b_{2}\gq_{\tau}(u) \ . $$
Moreover we have
$$ \|r u^{\rm sc}\|^2_{L^2_{r}(\R_{+})}=b_{2}^{-1}\|ru\|^2_{L^2_{r}(\R_{+})}  \quad \mbox{and} \quad \|\sqrt{r} u^{\rm sc}\|^2_{L^2_{r}(\R_{+})}=b_{2}^{-1/2}\|\sqrt{r}u\|^2_{L^2_{r}(\R_{+})}  \ ,$$
therefore we deduce \eqref{E:evalqm} from \eqref{E:evalqmlem}.
\end{proof}

\begin{prop}
Let $\bB\in \dS^2$ with $b_{2}>0$. There exists $C(\bB)>0$ such that 
\begin{equation}
\label{}
\forall \alpha\in (0,\pi), \quad E(\bB,\cW_{\alpha}) \leq b_{2}\Xi_{0}+C(\bB)\alpha^2 \ .
\end{equation}
\end{prop}
\begin{proof}
We recall that $z_{\tau}\in B^1_{r}(\R_{+})$ is a normalized eigenfunction associated to $\zeta(\tau)$ for the operator $\gg_{\tau}$ (see Subsection \ref{S:HL}). We define 
$$z_{\tau_{0}}^{\rm sc}(r):=b_{2}^{1/2}z_{\tau_{0}}(rb_{2}^{1/2}) $$
where $\tau_{0}\in \R$ satisfies $\zeta(\tau_{0})=\Xi_{0}$ (see \eqref{S:HL}). For all $\alpha>0$ we have: 
\begin{equation}
\label{E:ineqelementaires}
0\leq 1-\sinc\alpha\leq\frac{\alpha^2}{6} \quad \mbox{and}\quad 0\leq1-\sinc\frac{\alpha}{2} \leq \frac{\alpha^2}{24} \ . 
\end{equation}
We have $\gq_{\tau_{0}}(z_{\tau_{0}})=\Xi_{0}$, therefore Lemma \ref{L:evalqm} and \eqref{E:ineqelementaires} provides 
$$\FQpol_{\bB,\alpha}^{\tau_{0}\sqrt{b_{2}}}(z_{\tau_{0}}^{\rm sc})\leq b_{2}\Xi_{0}+\frac{\alpha^2}{12}\left(\frac{b_{3}^2+|b_{1}^2-b_{2}^2|}{b_{2}}\|rz_{\tau_{0}}\|^2_{L^2_{r}}+\tau_{0}b_{2}\|\sqrt{r}z_{\tau_{0}}\|^2_{L^2_{r}(\R_{+})}\right) \ .$$
Since $\|z_{\tau_{0}}^{\rm sc}\|_{L^2_{r}(\Omega_{0})}=\|z_{\tau_{0}}\|_{L^2_{r}(\R_{+})}=1$ the min-max principle provides:
$$\exists C(\bB)>0, \ \forall \alpha\in(0,\pi), \quad \sse(\bB,\cS_{\alpha};\tau_{0}\sqrt{b_{2}}) \leq b_{2}\Xi_{0}+C(\bB)\alpha^2 \ . $$
We deduce the proposition with \eqref{E:relationfond}.
\end{proof}
As a direct consequence we get 
\begin{cor}
\label{C:ineqsmallangle}
Let $\bB\in \dS^2$ with $b_{2}>0$. We have the following upper bound:
\begin{equation}
\label{Ub:lambdaxi0}
\limsup_{\alpha\to0}E(\bB,\cW_{\alpha}) \leq b_{2}\Xi_{0} \ .
\end{equation}
\end{cor}
Numerical computations show that $E(\bB,\cW_{\alpha})$ seems to go to $b_{2}\Xi_{0}$ when $\alpha$ goes to 0 (see Section \ref{S:numerique} and \cite[Section 6.4]{Popoff}). This question remains open. However the upper bound \eqref{Ub:lambdaxi0} is sufficient to give a comparison between the spectral quantity associated to an edge and the one coming from regular model problem:

\begin{theo}
\label{T:inegstricte}
Let $\bB\in \dS^2$ with $b_{2}>0$. Then there exists $\alpha(\bB)\in(0,\pi)$ such that for all $\alpha\in(0,\alpha(\bB))$ we have $E(\bB,\cW_{\alpha})<E^{*}(\bB,\cW_{\alpha})$. 
\end{theo}
\begin{proof}
We introduce $\theta^{0}:=\min\{\theta^{+},\theta^{-}\}$ ($\theta^{0}$ depends on $\alpha$ and $\bB$). For $\alpha\in (0,\pi)$ we have $E^{*}(\bB,\cW_{\alpha})=\sigma(\theta^{0})$. We recall the inequality from \cite[Section 3.4]{HeMo02}: 
$$\sigma(\theta^{0}) \geq \sqrt{\Theta_{0}^2\cos^2(\theta^{0})+\sin^2(\theta^{0})} \ . $$
Since $\theta^{0}$ goes to $\arcsin b_{2}$ when $\alpha$ goes to 0, we get 
$$\liminf_{\alpha\to0} E^{*}(\bB,\cW_{\alpha})=\liminf_{\alpha\to0} \sigma(\theta^{0}) \geq \sqrt{(1-\Theta_{0}^2)b_{2}^2+\Theta_{0}^2} \ . $$
Since $\Xi_{0}\in (0,1)$ (see Subsection \ref{E:inexi0}), we get: 
$$\forall b_{2}\in [0,1], \quad \Xi_{0}b_{2} < \sqrt{(1-\Theta_{0}^2)b_{2}^2+\Theta_{0}^2} $$
and we deduce from Corollary \ref{C:ineqsmallangle}:
$$\limsup_{\alpha\to0} E(\bB,\cW_{\alpha}) < \liminf_{\alpha\to0}E^{*}(\bB,\cW_{\alpha}) \ . $$
The theorem follows.
\end{proof}
\begin{rem}
It is possible to use gaussian quasimodes in \eqref{E:evalqm} and to deduce for $E(\bB,\cW_{\alpha})$ a polynomial in $\alpha$ upper bound with explicit constants (see \cite[Section 6.3]{Popoff}). This allows to get analytic value of $\alpha(\bB)$, for example we get with numerical approximations $\alpha(\bB) \geq 0.38\pi$ for the magnetic field $\bB=(0,1,0)$ normal to the plane of symmetry.
\end{rem}
\begin{rem}
The previous theorem remains true in the special case $b_{2}=0$ (see \cite[Section 7]{Popoff}) but the proof is different since the limit operator when $\alpha$ goes to 0 is not anymore the operator $\gq_{\tau}$ introduced in Section \ref{S:HL}.
\end{rem}

\section{Numerical simulations}
\label{S:numerique}
Let $C:=(0,L)^2$ be the square of length $L$. We perform a rotation by $-\frac{\pi}{4}$ around the origin and the scaling $X_{2}:=x_{2}\tan\frac{\alpha}{2}$ along the $x_{2}$-axis. The image of $C$ by these transformations is a rhombus of opening $\alpha$ denoted by $\Los(\alpha,L)$. The length of the diagonal supported by the $x_{1}$-axis is $\sqrt{2}L$. Using the finite element library M\'elina (\cite{Melina++}), we compute the first eigenpairs of $(-i\nabla-\Add)+V_{\uB}^{\tau}$ on $\Los(\alpha,L)$ with a Dirichlet condition on the artificial boundary $\{\partial\Los(\alpha,L)\cap \{x_{1}>\frac{1}{\sqrt{2}}L\}$. We denote by $\breve{\sse}(\bB,\cS_{\alpha};\tau)$ the numerical approximation of the first eigenvalue of this operator. For $L$ large, $\breve{\sse}(\bB,\cS_{\alpha};\tau)$ is a numerical approximation of $\sse(\bB,\cS_{\alpha};\tau)$. We refer to \cite[Annex C and Chapter 5]{Popoff} for more details about the meshes and the degree of the approximations we have used.

We make numerical simulations for the magnetic field $\bB=(\frac{1}{\sqrt{2}},\frac{1}{\sqrt{2}},0)$ which is normal to the edge. An associated linear potential is $\bA(x_{1},x_{2},x_{3})=(0,0,-\frac{x_{1}}{\sqrt{2}}+\frac{x_{2}}{\sqrt{2}})$ and we have 
$$H(\Add, \cS_{\alpha})+V_{\uB}^{\tau}=-\Delta+(\tfrac{x_{1}}{\sqrt{2}}-\tfrac{x_{2}}{\sqrt{2}}-\tau)^2 \ .$$
We notice that in that case the reduced operator on $\cS_{\alpha}$ is real and therefore its eigenfunctions have real values. For numerical simulations of eigenfunction with complex values, see \cite[Section 7]{Pof13T}.

\newpage
\begin{figure}[ht]
\begin{center}
\includegraphics[keepaspectratio=true,width=17cm]{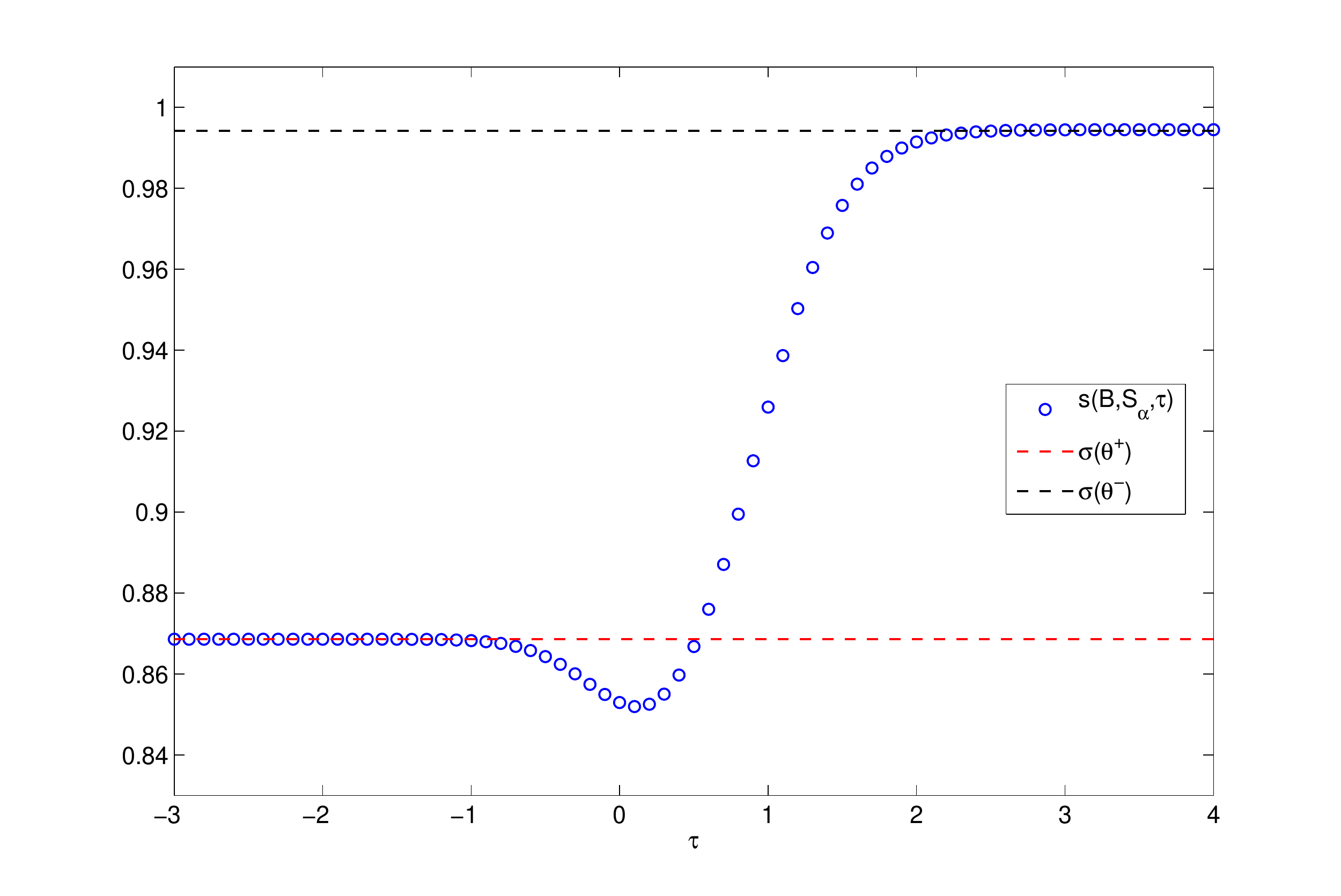}
\caption{Magnetic field: $\bB=(\frac{1}{\sqrt{2}},\frac{1}{\sqrt{2}},0)$. Opening angle: $\alpha=\frac{4\pi}{5}$. The numerical approximation of $\sse(\bB,\cS_{\alpha};\tau)$ versus $\tau$ compared with $\sigma(\theta^{+})$ and $\sigma(\theta^{-})$.}
\label{F1}
\end{center}
\end{figure}
On Figure \ref{F1} we have set $\alpha=\frac{4\pi}{5}$: the magnetic field is ingoing. In that case we have $\theta^{+}=\frac{3\pi}{20}$ and $\theta^{-}=\frac{7\pi}{20}$. We have shown $\breve{\sse}(\bB,\cS_{\alpha};\tau)$ for $\tau=\frac{k}{10}$ with $-30\leq k \leq 40$. We have also shown $\sigma(\theta^{+})$ and $\sigma(\theta^{-})$ where the numerical approximations of $\sigma(\cdot)$ comes from \cite{BoDauPopRay12}. $\breve{\sse}(\bB,\cS_{\alpha};\tau)$ seems to converge to $\sigma(\theta^{\mp})$ when $\tau$ goes to $\pm\infty$ in agreement with Proposition \ref{P:liminfTcase}. Moreover $\tau\mapsto \breve{\sse}(\bB,\cS_{\alpha};\tau)$ reaches its infimum and this infimum is strictly below $\sigma(\theta^{+})=E^{*}(\bB,\cW_{\alpha})$. Therefore we think that \eqref{E:ineqwedge} is strict for these values of $\bB$ and $\alpha$.

On figure \ref{F2} we show normalized eigenfunctions of $(-i\nabla-\Add)+V_{\uB}^{\tau}$ on $\Los(\frac{4\pi}{5},20)$ associated to $\breve{\sse}(\bB,\cS_{\frac{4\pi}{5}};\tau)$ for $\tau=k$, $-3\leq k \leq 4$. We see that the eigenfunctions are localized near the line $\Upsilon$ where the potential $V_{\uB}^{\tau}$ vanishes.

\newpage
\begin{figure}[!h]
\begin{center} 
\begin{tabular}{cccc}
\includegraphics[keepaspectratio=true,width=3cm]{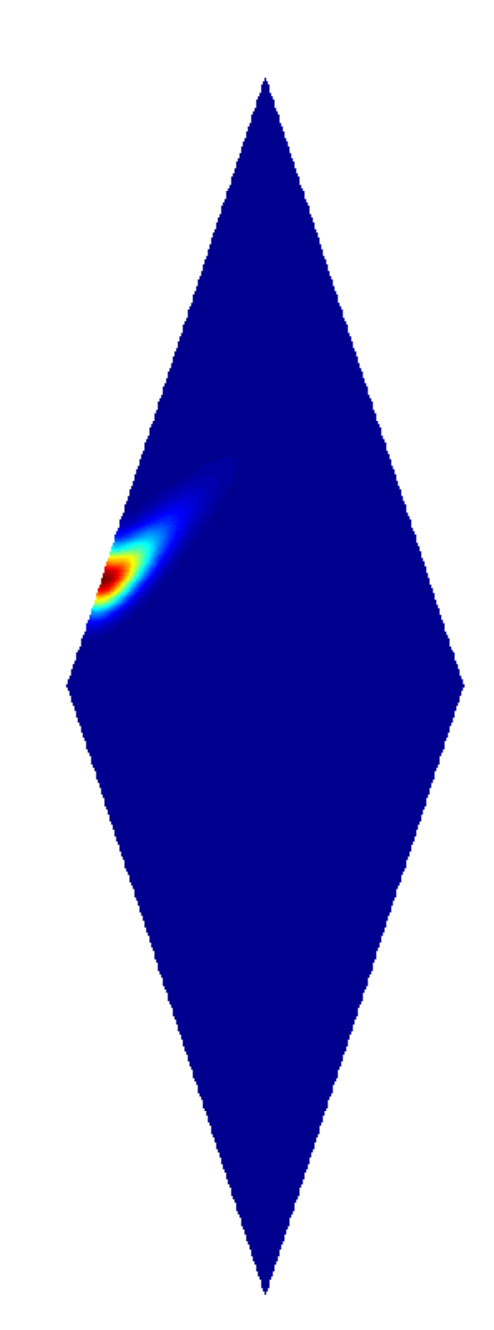}
&
\includegraphics[keepaspectratio=true,width=3cm]{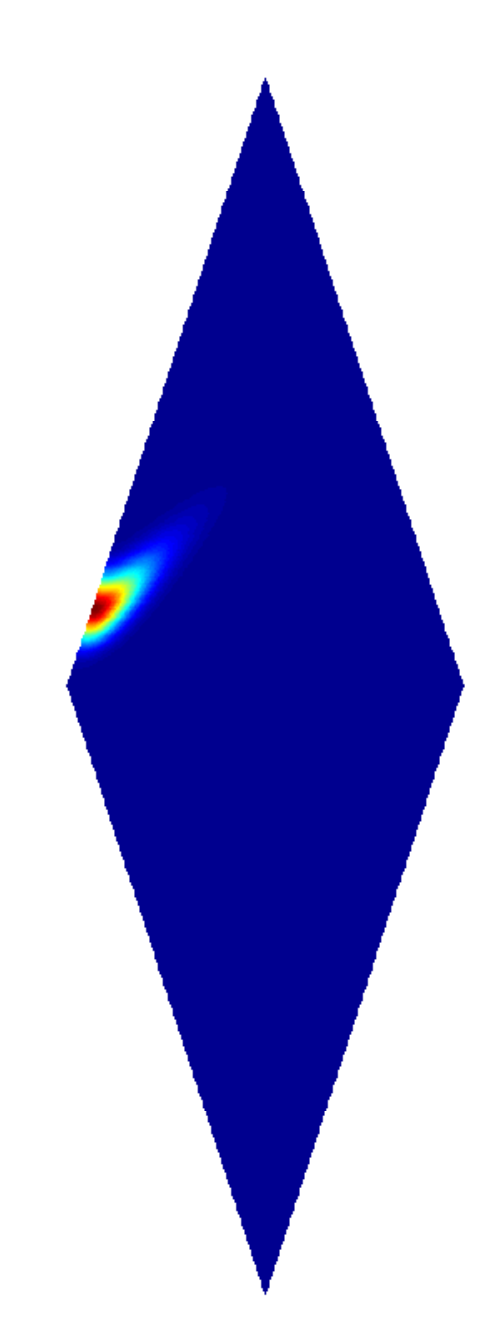} 
  &
\includegraphics[keepaspectratio=true,width=3cm]{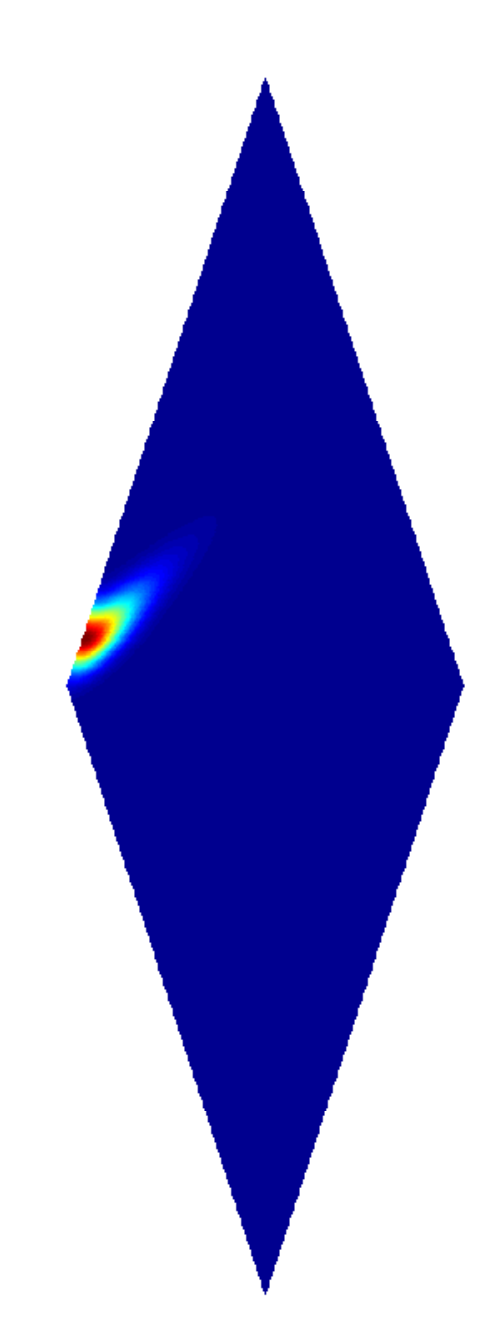}
&
\includegraphics[keepaspectratio=true,width=3cm]{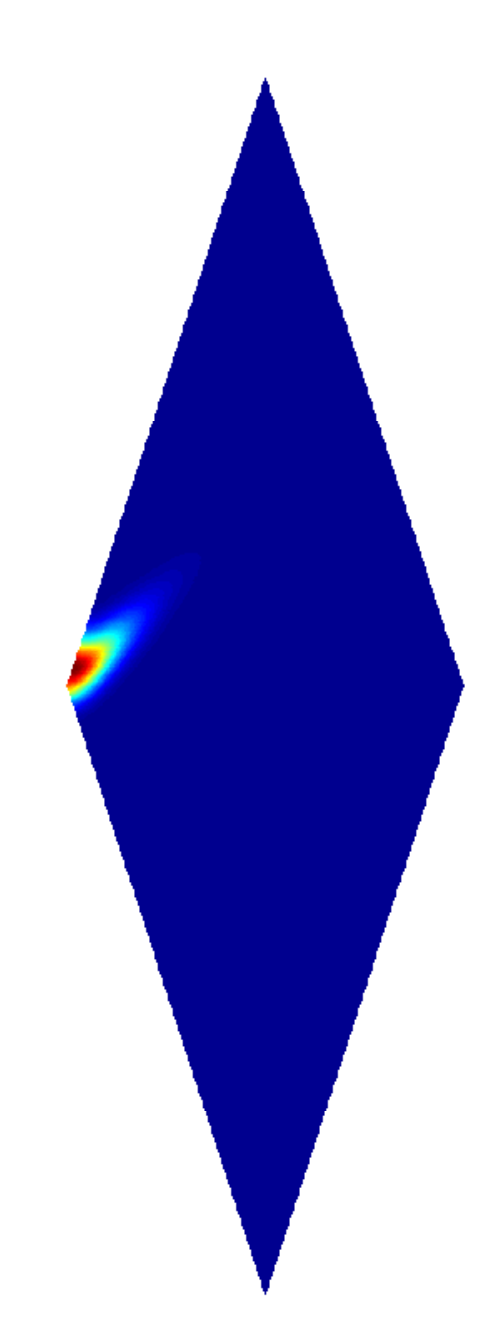} 
\\
$\tau=-3$ & $\tau=-2$ & $\tau=-1$ & $\tau=0$
\\
\includegraphics[keepaspectratio=true,width=3cm]{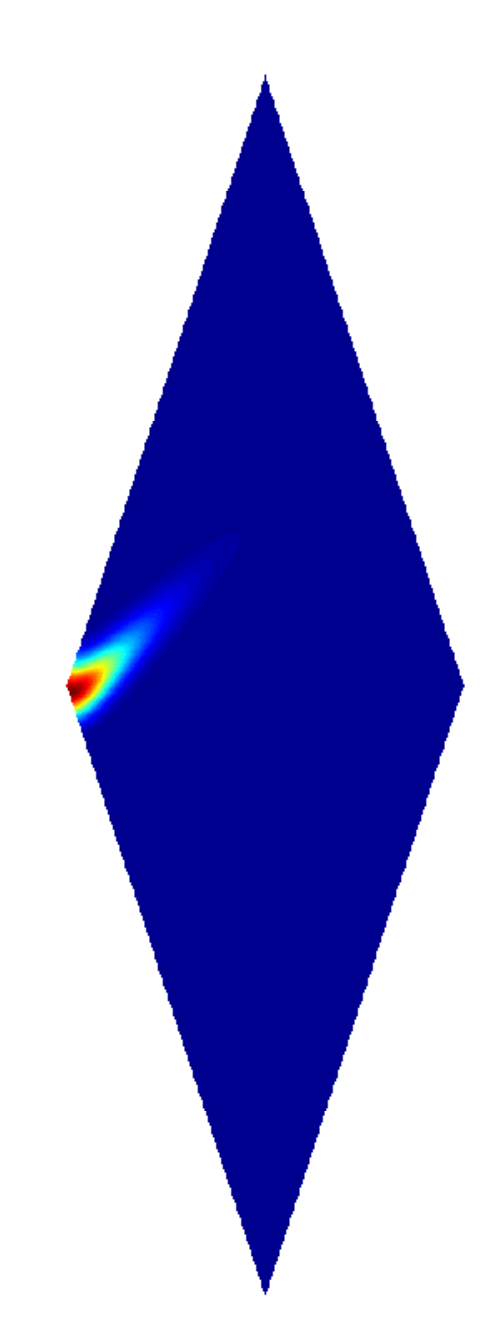}
&
\includegraphics[keepaspectratio=true,width=3cm]{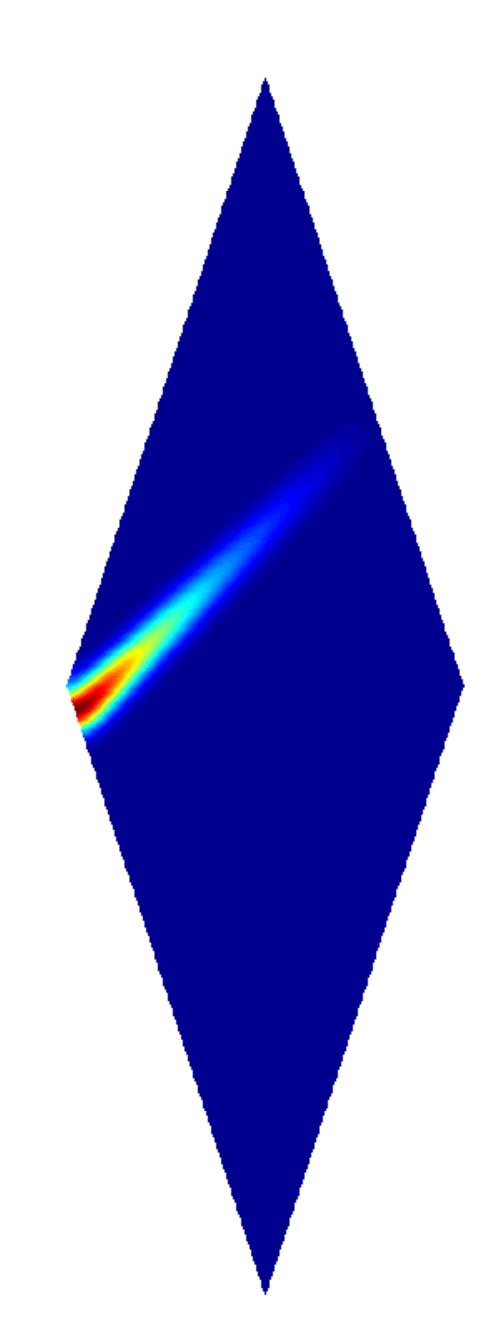}
&
\includegraphics[keepaspectratio=true,width=3cm]{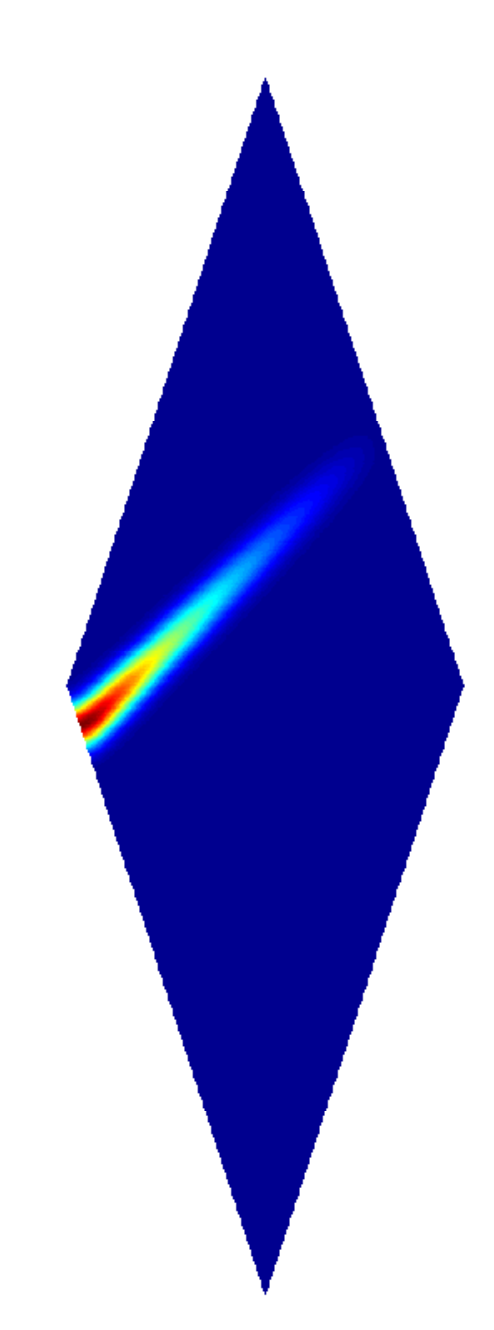}
&
\includegraphics[keepaspectratio=true,width=3cm]{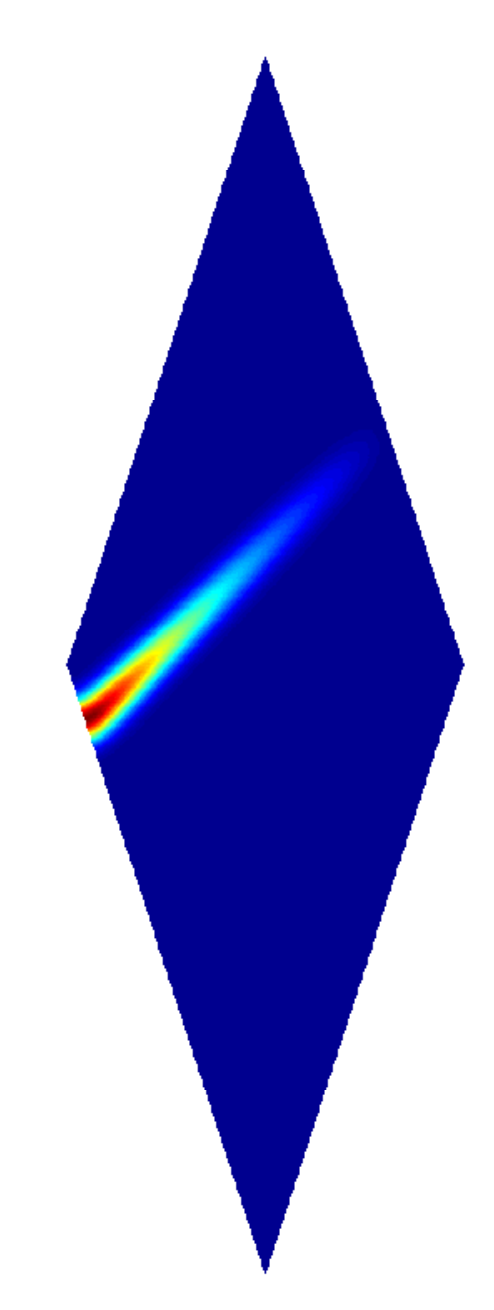}
\\
$\tau=1$ & $\tau=2$ & $\tau=3$ & $\tau=4$
\end{tabular}
\caption{Magnetic field: $\bB=(\frac{1}{\sqrt{2}},\frac{1}{\sqrt{2}},0)$. Opening angle: $\alpha=\frac{4\pi}{5}$. Normalized Eigenvectors of $H(\Add, \cS_{\alpha})+V_{\uB}^{\tau}$ associated to $\sse(\bB,\cS_{\alpha};\tau)$. From top to bottom and left to right: $\tau=k$, $-3 \leq k \leq 4$. Computational domain: $\Los(20,\frac{4\pi}{5})$.}
\label{F2}
\end{center}
\end{figure}

\newpage

\begin{figure}[ht!]
\begin{center}
\includegraphics[keepaspectratio=true,width=17cm]{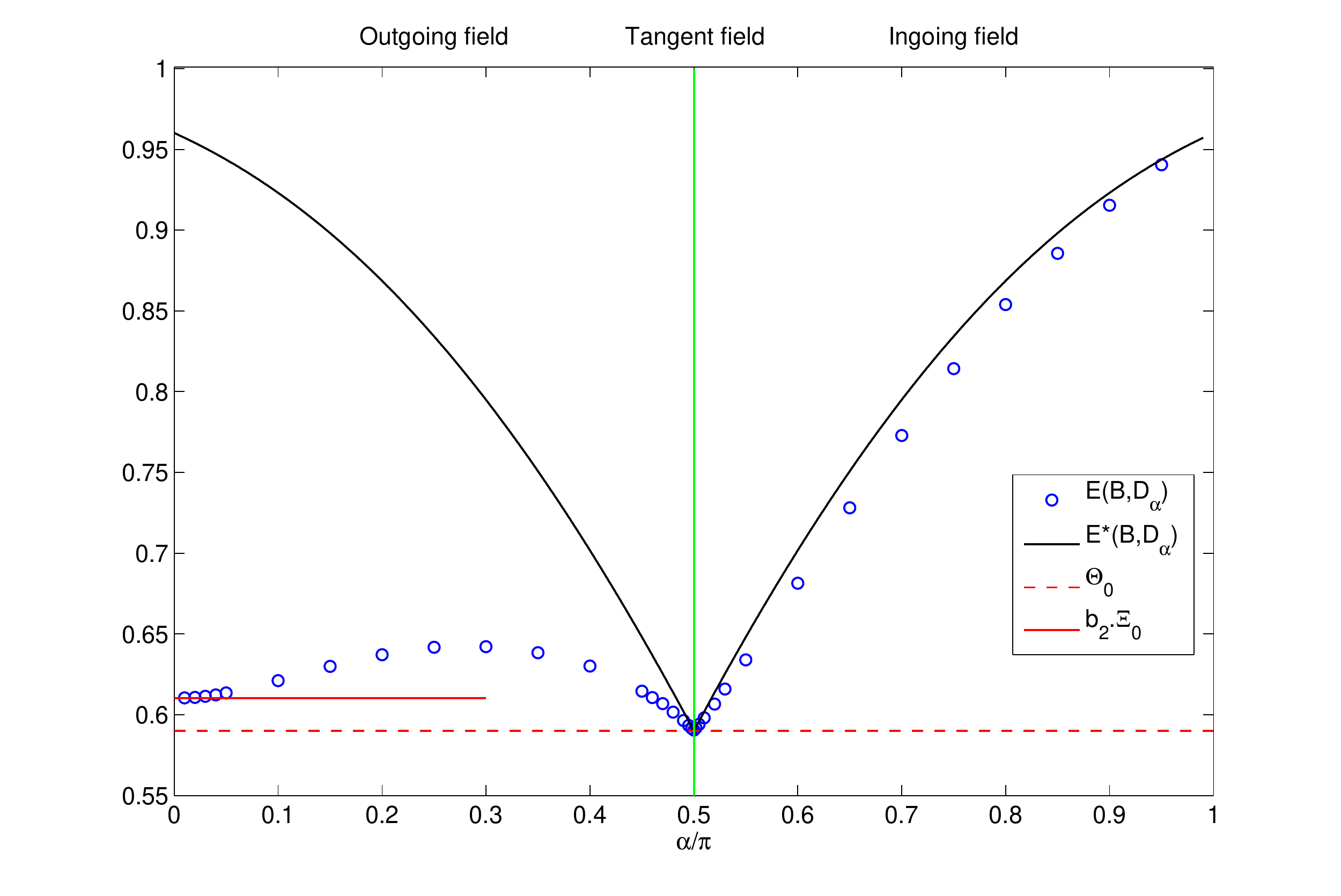}
\caption{Magnetic field: $\bB=(\frac{1}{\sqrt{2}},\frac{1}{\sqrt{2}},0)$. The numerical approximation of $E(\bB,\cW_{\alpha})$ versus $\frac{\alpha}{\pi}$ compared with $E^{*}(\bB,\cW_{\alpha})$, $b_{2}\Xi_{0}$ and $\Theta_{0}$.}
\label{F3}
\end{center}
\end{figure}
On figure \ref{F3} we show numerical approximations of $E(\bB,\cW_{\alpha})$. For each value of $\alpha$ we compute $\sse(\bB,\cS_{\alpha};\tau)$ for several values of $\tau$ and we define 
$$\breve{E}(\bB,\cW_{\alpha}):=\inf_{\tau}\breve{\sse}(\bB,\cS_{\alpha};\tau)$$
a numerical approximation of $E(\bB,\cW_{\alpha})$. The magnetic field is outgoing when $\alpha\in(0,\frac{\pi}{2})$, ingoing when $\alpha\in(\frac{\pi}{2},\pi)$ and tangent when $\alpha=\frac{\pi}{2}$. We notice that $\breve{E}(\bB,\cW_{\alpha})$ seems to converge to $b_{2}\Xi_{0}$ (see Subsection \ref{SS:upperboundsmallangles}). We have also plotted $E^{*}(\bB,\cW_{\alpha})$ according to \eqref{E:lambdastarexplicit} and to the numerical values of $\sigma(\cdot)$ coming from \cite{BoDauPopRay12}. We see that for $\alpha\neq\frac{\pi}{2}$, we have $\breve{E}(\bB,\cW_{\alpha})<E^{*}(\bB,\cW_{\alpha})$ whereas $\breve{E}(\bB,\cW_{\frac{\pi}{2}})\approx \Theta_{0}=E^{*}(\bB,\cW_{\frac{\pi}{2}})$. Let us also notice that $\alpha\mapsto E(\bB,\cW_{\alpha})$ seems not to be $\mathcal{C}^{1}$ in $\alpha=\frac{\pi}{2}$.

\paragraph{Aknowledgement}
Some of this work is part of a PhD. thesis. The author would like to thank M. Dauge and V. Bonnaillie-No\"el for sharing lot of ideas and for useful discussions. He is also grateful to M. Dauge for careful reading and suggestions.

\newpage
\newpage
\bibliographystyle{mnachrn}

\begin{thebibliography}{10}

\bibitem{Ag82}
{\sc S.~Agmon}.
\newblock {\em Lectures on exponential decay of solutions of second-order
  elliptic equations: bounds on eigenfunctions of {$N$}-body {S}chr\"odinger
  operators}, volume~29 of {\em Mathematical Notes}.
\newblock Princeton University Press, Princeton, NJ 1982.

\bibitem{BeSt}
{\sc A.~Bernoff, P.~Sternberg}.
\newblock Onset of superconductivity in decreasing fields for general domains.
\newblock {\em J. Math. Phys.} {\bf 39}(3) (1998) 1272--1284.

\bibitem{BolHe93}
{\sc C.~Bolley, B.~Helffer}.
\newblock An application of semi-classical analysis to the asymptotic study of
  the supercooling field of a superconducting material.
\newblock {\em Ann. Inst. H. Poincar\'e Phys. Th\'eor.} {\bf 58}(2) (1993)
  189--233.

\bibitem{BolCa72}
{\sc P.~Bolley, J.~Camus}.
\newblock Sur une classe d'op\'erateurs elliptiques et d\'eg\'en\'er\'es \`a
  une variable.
\newblock {\em J. Math. Pures Appl. (9)} {\bf 51} (1972) 429--463.

\bibitem{Bon06}
{\sc V.~Bonnaillie}.
\newblock On the fundamental state energy for a {S}chr\"odinger operator with
  magnetic field in domains with corners.
\newblock {\em Asymptot. Anal.} {\bf 41}(3-4) (2005) 215--258.

\bibitem{BonDau06}
{\sc V.~Bonnaillie-No\"el, M.~Dauge}.
\newblock Asymptotics for the low-lying eigenstates of the {S}chr\"odinger
  operator with magnetic field near corners.
\newblock {\em Ann. Henri Poincar\'e} {\bf 7} (2006) 899--931.

\bibitem{BoDauMaVial07}
{\sc V.~Bonnaillie-No{\"e}l, M.~Dauge, D.~Martin, G.~Vial}.
\newblock Computations of the first eigenpairs for the {S}chr\"odinger operator
  with magnetic field.
\newblock {\em Comput. Methods Appl. Mech. Engrg.} {\bf 196}(37-40) (2007)
  3841--3858.

\bibitem{BoDauPof13}
{\sc V.~Bonnaillie-No{\"e}l, M.~Dauge, N.~Popoff}.
\newblock Polyhedral bodies in large magnetic fields.
\newblock {\em Ongoing work}  (2013).

\bibitem{BoDauPopRay12}
{\sc V.~Bonnaillie-No{\"e}l, M.~Dauge, N.~Popoff, N.~Raymond}.
\newblock Discrete spectrum of a model {S}chr\"odinger operator on the
  half-plane with {N}eumann conditions.
\newblock {\em ZAMP} {\bf 63}(2) (2012) 203--231.

\bibitem{CyFrKiSi87}
{\sc H.~Cycon, R.~Froese, W.~Kirsch, B.~Simon}.
\newblock {\em Schr\"odinger operators with application to quantum mechanics
  and global geometry}.
\newblock Texts and Monographs in Physics. Springer-Verlag, Berlin, study
  edition 1987.

\bibitem{Dau88}
{\sc M.~Dauge}.
\newblock {\em Elliptic boundary value problems on corner domains}, volume 1341
  of {\em Lecture Notes in Mathematics}.
\newblock Springer-Verlag, Berlin 1988.
\newblock Smoothness and asymptotics of solutions.

\bibitem{DauHe93}
{\sc M.~Dauge, B.~Helffer}.
\newblock Eigenvalues variation. {I}. {N}eumann problem for {S}turm-{L}iouville
  operators.
\newblock {\em J. Differential Equations} {\bf 104}(2) (1993) 243--262.

\bibitem{FouHel06}
{\sc S.~Fournais, B.~Helffer}.
\newblock Accurate eigenvalue estimates for the magnetic {N}eumann {L}aplacian.
\newblock {\em Annales Inst. Fourier} {\bf 56}(1) (2006) 1--67.

\bibitem{GeNi98}
{\sc C.~G{\'e}rard, F.~Nier}.
\newblock The {M}ourre theory for analytically fibered operators.
\newblock {\em J. Funct. Anal.} {\bf 152}(1) (1998) 202--219.

\bibitem{HeMo01}
{\sc B.~Helffer, A.~Morame}.
\newblock Magnetic bottles in connection with superconductivity.
\newblock {\em J. Funct. Anal.} {\bf 185}(2) (2001) 604--680.

\bibitem{HeMo02}
{\sc B.~Helffer, A.~Morame}.
\newblock Magnetic bottles for the {N}eumann problem: the case of dimension 3.
\newblock {\em Proc. Indian Acad. Sci. Math. Sci.} {\bf 112}(1) (2002) 71--84.
\newblock Spectral and inverse spectral theory (Goa, 2000).

\bibitem{HeMo04}
{\sc B.~Helffer, A.~Morame}.
\newblock Magnetic bottles for the {N}eumann problem: curvature effects in the
  case of dimension 3 (general case).
\newblock {\em Ann. Sci. \'Ecole Norm. Sup. (4)} {\bf 37}(1) (2004) 105--170.

\bibitem{Ja01}
{\sc H.~Jadallah}.
\newblock The onset of superconductivity in a domain with a corner.
\newblock {\em J. Math. Phys.} {\bf 42}(9) (2001) 4101--4121.

\bibitem{kato}
{\sc T.~Kato}.
\newblock {\em Perturbation theory for linear operators}.
\newblock Classics in Mathematics. Springer-Verlag, Berlin 1995.
\newblock Reprint of the 1980 edition.

\bibitem{LuPan99-2}
{\sc K.~Lu, X.-B. Pan}.
\newblock Eigenvalue problems of {G}inzburg-{L}andau operator in bounded
  domains.
\newblock {\em J. Math. Phys.} {\bf 40}(6) (1999) 2647--2670.

\bibitem{LuPan00}
{\sc K.~Lu, X.-B. Pan}.
\newblock Surface nucleation of superconductivity in 3-dimensions.
\newblock {\em J. Differential Equations} {\bf 168}(2) (2000) 386--452.

\bibitem{Melina++}
{\sc D.~Martin}.
\newblock M\'elina, biblioth\`eque de calculs \'el\'ements finis.
\newblock {\em http\char58//anum-maths.univ-rennes1.fr/melina}  (2010).

\bibitem{MazyaPlamenevskii77}
{\sc V.~G. Maz'ya, B.~A. Plamenevskii}.
\newblock Elliptic boundary value problems on manifolds with singularities.
\newblock {\em Probl. Mat. Anal.} {\bf 6} (1977) 85--142.

\bibitem{Pan02}
{\sc X.-B. Pan}.
\newblock Upper critical field for superconductors with edges and corners.
\newblock {\em Calc. Var. Partial Differential Equations} {\bf 14}(4) (2002)
  447--482.

\bibitem{Pers60}
{\sc A.~Persson}.
\newblock Bounds for the discrete part of the spectrum of a semi-bounded
  {S}chr\"odinger operator.
\newblock {\em Math. Scand.} {\bf 8} (1960) 143--153.

\bibitem{Popoff}
{\sc N.~Popoff}.
\newblock {\em {Sur l'op\'erateur de Schr¬\"odinger magn\'etique dans un
  domaine di\'edral}}.
\newblock PhD thesis, University of Rennes 1 2012.

\bibitem{Pof13-II}
{\sc N.~Popoff}.
\newblock On the lowest energy of a 3d magnetic laplacian with axisymmetric
  potential.
\newblock {\em Preprint IRMAR}  (2013).

\bibitem{Pof13T}
{\sc N.~Popoff}.
\newblock The {S}chr{\"o}dinger operator on a wedge with a tangent magnetic
  field.
\newblock {\em J. Math. Phys.} {\bf 54}(4) (2013).

\bibitem{PofRay13}
{\sc N.~Popoff, N.~Raymond}.
\newblock When the 3d-magnetic laplacian meets a curved edge in the
  semi-classical limit.
\newblock {\em SIAM J. Math. Anal.}  (To appear).

\bibitem{Ray3d09}
{\sc N.~Raymond}.
\newblock On the semi-classical 3{D} {N}eumann {L}aplacian with variable
  magnetic field.
\newblock {\em Asymptotic Analysis} {\bf 68}(1-2) (2010) 1-- 40.

\bibitem{ReSi78}
{\sc M.~Reed, B.~Simon}.
\newblock {\em Methods of modern mathematical physics. {IV}. {A}nalysis of
  operators}.
\newblock Academic Press [Harcourt Brace Jovanovich Publishers], New York 1978.

\bibitem{SaGe63}
{\sc D.~{Saint-James}, P.-G. {de Gennes}}.
\newblock Onset of superconductivity in decreasing fields.
\newblock {\em Physics Letters} {\bf 7} (Dec. 1963) 306--308.

\bibitem{Yaf08}
{\sc D.~Yafaev}.
\newblock On spectral properties of translationally invariant magnetic
  {S}chr\"odinger operators.
\newblock {\em Ann. Henri Poincar\'e} {\bf 9}(1) (2008) 181--207.

\end{thebibliography}

\end{document}